\documentclass{amsart}
\usepackage{amssymb}
\usepackage{mathrsfs}
\usepackage[shortalphabetic]{amsrefs}
\usepackage{array}
\usepackage{footnote}
\usepackage{ulem}
\usepackage{amsmath}
\usepackage{bbm}
\usepackage[shortlabels]{enumitem}
\usepackage{caption}
\usepackage{tikz-cd}
\makeatletter
\tikzcdset{
    cong/.style={"\cong" {sloped, description, yshift=0pt,#1}, phantom},
    simeq/.style={"\simeq" {sloped, description, yshift=0pt,#1}, phantom},
    snake/.style={
        out= east, in=west,
        to path={
            \pgfextra{
                \pgfextractx{\pgf@xa}{\pgfpointanchor{\tikztostart}{east}}
                \pgfextractx{\pgf@xb}{\pgfpointanchor{\tikztotarget}{west}}
                \pgfextracty{\pgf@ya}{\pgfpointanchor{\tikztostart}{center}}
                \pgfextracty{\pgf@yb}{\pgfpointanchor{\tikztotarget}{center}}
                \edef\tikzstartx{\the\pgf@xa}
                \edef\tikzendx{\the\pgf@xb}
                \edef\midy{\the\dimexpr0.5\dimexpr\pgf@ya\relax +0.5\dimexpr\pgf@yb\relax}
            }
            to[in=0,out=180,looseness=0.5] (\tikzstartx,\midy)
            -| ([xshift=-2ex]\tikztotarget.west)
            -- (\tikztotarget)}
    }
}
\makeatother
\usepackage{tikzcdintertext}

\usepackage[all]{xy}
\usepackage{todonotes}
\usepackage{upgreek}
\usepackage{verbatim}

\usepackage{multicol}
\usepackage{wrapfig}

\usepackage{longtable,booktabs}

\usepackage[parfill]{parskip}
\usepackage{hyperref}  
\hypersetup{%
  bookmarksnumbered=true,%
  bookmarks=true,%
  colorlinks=true,%
  linkcolor=blue,%
  citecolor=blue,%
  filecolor=blue,%
  menucolor=blue,%
  pagecolor=blue,%
  urlcolor=blue,%
  pdfnewwindow=true,%
  pdfstartview=FitBH}

\newcommand{\cA}{\mathcal{A}}

\newcommand {\M}{\mathrm{M}}

\newcommand{\C}{\mathbb{C}}
\newcommand{\cC}{\mathcal{C}}

\newcommand{\F}{\mathbb{F}_2}

\newcommand{\bM}{\mathbb{M}_2}
\newcommand{\bN}{\mathbb{N}}
\newcommand{\cQ}{\mathcal{Q}}
\newcommand{\R}{\mathbb{R}}
\newcommand{\bS}{\mathbb{S}}
\newcommand{\cY}{\mathcal{Y}}
\newcommand{\Z}{\mathbb{Z}}

\newcommand{\rH}{\mathrm{H}}

\newcommand{\xrtarr}[1]{\stackrel{#1}{\longrightarrow}}
\newcommand{\rtarr}{\longrightarrow}

\newcommand{\iso}{\cong}
\newcommand{\sma}{\wedge}

\newcommand{\hsf}{\mathsf{h}}

\newcommand{\mm}{/\!\!/}

\newcommand{\Sp}{\mathbf{Sp}} 
\newcommand{\RP}{\mathbb{RP}}
\newcommand{\CP}{\mathbb{CP}}

\newcommand{\mr}[1]{\mathrm{#1}}

\newcommand{\fin}{\mr{fin}}
\newcommand{\Betti}{\mr{\upbeta}}
\newcommand{\Ctwo}{{\mr{C}_2}}
\newcommand{\nil}{\mr{nil}}
 
\newcommand{\Y}{\mathcal{Y}}
\newcommand{\QR}{\cQ_{\R}}
\newcommand{\QC}{\cQ_{\Ctwo}}
\newcommand{\coneR}[1]{\mr{C}^\R(#1)}
\newcommand{\coneC}[1]{\mr{C}^{\Ctwo}(#1)}
\renewcommand{\emph}[1]{\textit{#1}}

\newcommand{\Ytriv}{\Y_{\mr{triv} } }
\newcommand{\Yone}{\Y_{(2,1)} }
\newcommand{\Ytwo}{\Y_{(\hsf,0)} } 
\newcommand{\Ythree}{\Y_{(\hsf, 1)} } 

\DeclareMathOperator{\Ext}{Ext}

\newcommand{\bMR}{\bM^\R}

\newcommand{\cAR}{\cA^\R}
\newcommand{\cAC}{\cA^{\Ctwo}}

\newcommand{\Sq}{\mathrm{Sq}}
\newcommand{\ho}{\mr{Ho}}
\newif\ifCharts    
\Chartstrue

\numberwithin{equation}{section}
\numberwithin{figure}{section}

\def\makeautorefname#1#2{\expandafter\def\csname#1autorefname\endcsname{#2}}
\makeautorefname{theorem}{Theorem}%
\makeautorefname{thm}{Theorem}%
\makeautorefname{prop}{Proposition}%
\makeautorefname{addm}{Addendum}%
\makeautorefname{mainthm}{Main theorem}%
\makeautorefname{corollary}{Corollary}%
\makeautorefname{const}{Construction}%
\makeautorefname{cor}{Corollary}%
\makeautorefname{lemma}{Lemma}%
\makeautorefname{notn}{Notation}%
\makeautorefname{rmk}{Remark}%
\makeautorefname{lem}{Lemma}%
\makeautorefname{probl}{Problem}%
\makeautorefname{section}{Section}%
\makeautorefname{subsection}{Subsection}%
\makeautorefname{sublemma}{Sublemma}%
\makeautorefname{sublem}{Sublemma}%
\makeautorefname{subl}{Sublemma}%
\makeautorefname{prop}{Proposition}%
\makeautorefname{warn}{Warning}%
\makeautorefname{eg}{Example}%
\makeautorefname{figure}{Figure}%
\makeautorefname{qst}{Question}%

\newtheorem{thm}{Theorem}[section]
\newtheorem{cor}{Corollary}[section]
\newtheorem{prop}{Proposition}[section]
\newtheorem{lem}{Lemma}[section]

\theoremstyle{definition}
\newtheorem{defn}{Definition}[section]

\newtheorem{eg}{Example}[section]
\newtheorem{notn}{Notation}[section]

\newtheorem{rmk}{Remark}[section]

\newenvironment{pf}{\begin{proof}}{\end{proof}}

\makeatletter
\let\c@cor=\c@thm
\let\c@con=\c@thm
\let\c@prop=\c@thm
\let\c@lem=\c@thm
\let\c@defn=\c@thm
\let\c@eg=\c@thm
\let\c@notn=\c@thm
\let\c@rmk=\c@thm
\let\c@warn=\c@thm
\let\c@qst=\c@thm
\let\c@probl=\c@thm
\let\c@equation=\c@thm
\let\c@figure=\c@thm
\makeatother

\newcolumntype{L}{>{$}l<{$}}

\newenvironment{psmallmatrix}
  {\left(\begin{smallmatrix}}
  {\end{smallmatrix}\right)}

\newcounter{themyfigure}

\newcommand{\slfmp}{{--self-map}}

\title{An $\R$-motivic $v_{1}$%
--self-map 
of periodicity $1$}

\author{P. Bhattacharya}
\address{Department of Mathematics,
University of Notre Dame,
Notre Dame, IN  46556} 
\email{pbhattac@nd.edu}
\author{B. Guillou}
\address{Department of Mathematics, The University of Kentucky, Lexington, KY 40506--0027}
\email{bertguillou@uky.edu}
\author{A. Li}
\address{Department of Mathematics, The University of Kentucky, Lexington, KY 40506--0027}
\email{ang.li1414201@uky.edu}

\thanks{B. Guillou and A. Li were  supported by NSF grants DMS-1710379 and DMS-2003204.}

\begin{document}

\begin{abstract}
We consider a nontrivial action of $\mathrm{C}_2$ on the type $1$ spectrum $\mathcal{Y} := \mr{M}_2(1) \sma \mr{C}(\eta)$, which is well-known for admitting a $1$-periodic $v_1$--self-map. The resultant finite $\Ctwo$-equivariant spectrum $\mathcal{Y}^{\mathrm{C}_2}$ can also be viewed as the complex points of  a finite $\mathbb{R}$-motivic spectrum $\mathcal{Y}^\mathbb{R}$. In this paper, we show that one of the $1$-periodic $v_1$--self-maps of $\mathcal{Y}$ can be lifted to  a self-map of $\mathcal{Y}^{\mathrm{C}_2}$ as well as   $\mathcal{Y}^{\mathbb{R}}$. Further, the cofiber of the  self-map of $\mathcal{Y}^{\mathbb{R}}$ is a realization of the subalgebra $\mathcal{A}^\R(1)$ of the $\R$-motivic Steenrod algebra.  We also show that the $\mathrm{C}_2$-equivariant self-map is  nilpotent on the geometric fixed-points of $\mathcal{Y}^{\mathrm{C}_2}$.
\end{abstract}

\maketitle

\section{Introduction}

In  classical stable homotopy theory, the interest in periodic $v_n$\slfmp{s} of finite spectra lies in the fact that one can associate to each $v_n$\slfmp\ an infinite family in the chromatic layer $n$  stable homotopy groups of spheres. Therefore, interest lies in constructing type $n$ spectra and finding $v_n$\slfmp{s} of lowest possible periodicity on a given type $n$ spectrum. This, in general, is a difficult problem, though progress has been made  sporadically throughout the history of the subject \cites{T71,DM, BP, BHHM, Nave, BEM, BE20}. With the modern development of motivic stable homotopy theory, one may ask if 
there are
 similar  periodic self-maps of finite motivic spectra. 

 Classically any non-contractible finite $p$-local spectrum admits a periodic $v_n$\slfmp \   for some $n \geq 0$.  This is a consequence of the thick-subcategory theorem \cite[Theorem 7]{Thick}, aided by a vanishing line argument \cite[\S 4.2]{Thick}. 
 In the classical case all the thick tensor ideals of $\Sp_{p,\fin}$ (the homotopy category of finite $p$-local spectra) are also prime (in the sense of \cite{B}). The  thick tensor-ideals of the homotopy category of cellular motivic spectra over $\C$ or $\R$ are 
not completely known (but see \cite{HelOrm}). However, one can gather some knowledge about the prime thick tensor-ideals  in $\ho(\Sp^{\R}_{2,\fin})$ (the homotopy category of $2$-local cellular $\R$-motivic spectra) through the Betti realization functor 
 \[ 
\begin{tikzcd}
\Betti: \ho(\Sp^\R_{2,\fin}) \rar & \ho(\Sp^{\mr{C}_2}_{2,\fin})
\end{tikzcd}
\]
using the complete knowledge of prime thick subcategories of $\ho(\Sp^{\mr{C}_2}_{2,\fin})$ \cite{BaSa}.

The prime thick tensor-ideals of $\ho(\Sp^{\mr{C}_2}_{2,\fin})$  are essentially the pull-back of the classical thick subcategories  along the two functors, the geometric fix point functor  
 \[ 
\begin{tikzcd}
\Phi^{\mr{C}_2} : \ho(\Sp^{\mr{C}_2}_{2,\fin}) \rar & \ho(\Sp_{2,\fin})
\end{tikzcd}
\]
and the forgetful functor 
\[ 
\begin{tikzcd}
\Phi^{e} : \ho(\Sp^{\mr{C}_2}_{2,\fin}) \rar & \ho(\Sp_{2,\fin}).
\end{tikzcd}
\]
Let $\mathcal{C}_n$ denote the thick subcategory of $\ho(\Sp_{2,\fin})$ consisting of spectra of type at least $n$. The prime thick subcategories, 
\[ \mathcal{C}( e, n) =( \Phi^{e})^{-1}(\mathcal{C}_n) \text{ and } \mathcal{C}( \Ctwo, n) = (\Phi^{\mr{C}_2})^{-1}(\mathcal{C}_n),\] are the only prime thick subcategories of $\ho(\Sp^{\mr{C}_2}_{2,\fin})$. 

\begin{defn}
We say a spectrum $X \in \ho(\Sp^{\mr{C}_2}_{2,\fin})$ is of \emph{ type $(n,m)$ }iff $\Phi^e(X)$ is of type $n$ and  $\Phi^{\mr{C}_2}(X)$ is of type $m$. 
\end{defn}

For a type $(n,m)$ spectrum $X$, a self-map $f: X \to X$  is periodic if and only if  at least one of $\{ \Phi^e(f) ,\Phi^{\mr{C}_2}(f)\}$ are periodic (see \cite[Proposition 3.17]{BGH}). 

\begin{defn} Let $X \in \ho(\Sp^{\mr{C}_2}_{2,\fin})$ be of type $(n,m)$.  We say a self-map $f:X \to X $ is
\begin{enumerate}[(i)]
\item  a \emph{ $v_{(n,m)}$\slfmp} of mixed periodicity $(i,j)$ if $\Phi^{e}(f)$ is a  $v_n$\slfmp{} of periodicity $i$ and $\Phi^{\mr{C}_2}(f)$ is 
a $v_m$\slfmp\ of periodicity $j$, 
\item   a \emph{$v_{(n, \mr{nil})}$\slfmp} of periodicity $i$ if $\Phi^e(f)$ is a $v_n$\slfmp\ of periodicity $i$ and  $\Phi^{\mr{C}_2}(f)$ is nilpotent, and, 
\item  a \emph{$v_{( \mr{nil}, m)}$\slfmp} of periodicity $j$ if $\Phi^e(f)$ is a nilpotent self-map and   $\Phi^{\mr{C}_2}(f)$ is a $v_m$\slfmp\ of periodicity $j$. 
\end{enumerate}
\end{defn}

\begin{eg} \label{egCtwo1}
 The sphere spectrum $\bS_{\Ctwo}$ is of  type $(0,0)$. The 
degree 2 map
  is a $v_{(0,0)}$\slfmp. In general, if we consider the $v_n$\slfmp\ of a type $n$ spectrum with trivial action of $\Ctwo$, then the resultant equivariant self-map is a $v_{(n,n)}$\slfmp. 
\end{eg}

\begin{eg}  \label{egCtwo3}
Let $\mr{S}^{1,1}_{\Ctwo}$ denote the $\Ctwo$-equivariant sphere  which is the one-point compactification of the real sign representation.  The  unstable twist-map 
\[ 
\begin{tikzcd}
\epsilon_u: \mr{S}^{1,1}_{\Ctwo} \sma \mr{S}^{1,1}_{\Ctwo}  \rar & \mr{S}^{1,1}_{\Ctwo}  \sma \mr{S}^{1,1}_{\Ctwo} 
\end{tikzcd}
\]
stabilizes to a nonzero element  $\epsilon \in \pi_{0,0}(\bS_{\Ctwo})$. Let $\hsf =1- \epsilon \in \pi_{0,0}(\bS_{\Ctwo})$ be the stabilization of the map 
\[ 
\begin{tikzcd}
\hsf_u = 1 - \epsilon_u: \mr{S}^{3,2}_{\Ctwo}   \rar & \mr{S}^{3,2}_{\Ctwo}.
\end{tikzcd}
\]
Note that on the underlying space $\epsilon$ is of  degree $-1$, while on the fixed points it is the identity. Therefore $\Phi^e(\hsf)$ is multiplication by $2$, whereas $\Phi^{\Ctwo}(\hsf)$ is trivial. Thus $\hsf$ is a $v_{(0,\nil)}$\slfmp. Thus 
$\coneC{\hsf}$ is of type $(1,0)$.
\end{eg}

\begin{eg}  \label{egCtwo2}
The equivariant Hopf-map $\eta_{1,1} \in \pi_{1,1}(\bS_{\Ctwo})$
is the Betti realization of the $\R$-motivic Hopf-map $\eta$  \cites{Morel,MotivicHopf}.
Up to a unit, it is the stabilization of the projection map  
\[ 
\begin{tikzcd}
\eta_{1,1}^{u}:=  \pi: \mr{S}^{3,2}_{\Ctwo}  \simeq \C^2 \setminus\{  \mathbf{0} \}   \rar & \CP^1 \iso \mr{S}^{2,1}_{\Ctwo},
\end{tikzcd} \]
where the domain and the codomain are given the $\Ctwo$-structure using  complex conjugation. On  fixed-points, the map $\pi$ is the projection map 
\[ 
\begin{tikzcd}
\pi: \R^2 \setminus\{  \mathbf{0} \} \rar & \RP^1,
\end{tikzcd} \]
which is a degree $2$ map. From this we learn that while $\Phi^e(\eta_{1,1})$ is nilpotent, $\Phi^{\Ctwo}(\eta_{1,1})$ is the periodic $v_0$\slfmp. Hence, $\eta_{1,1}$ is a $v_{(\nil,0)}$\slfmp\ and the cofiber $\mr{C}(\eta_{1,1})$ is of type $(0,1)$. 
\end{eg}

\begin{rmk} In the $\Ctwo$-equivariant stable homotopy groups, the usual Hopf-map (sometimes referred to as the `topological Hopf-map') is different from $\eta_{1,1}$ of \autoref{egCtwo2}.  The `topological Hopf-map' $\eta_{1,0} \in \pi_{1,0}(\bS_{\Ctwo}) $ should be thought of as the stabilization of the unstable Hopf-map 
\[ 
\begin{tikzcd}
(\eta_{1,0})_u \colon \mr{S}^{3,0}_{\Ctwo} \rar & \mr{S}^{2,0}_{\Ctwo} 
\end{tikzcd}
\]
where both domain and codomain are given the trivial $\Ctwo$-action. 
\end{rmk}

\begin{defn} We say a spectrum $X \in \ho(\Sp_{2,\fin}^\R)$ is of type $(n,m)$ if $\Betti (X)$ is of type $(n,m)$. We call an $\R$-motivic self-map 
\[ f: X \to X\]
 a $v_{(n,m)}$\slfmp, where 
 $m$ and $n$ are in $\bN \cup \{ \nil \}$ (but not both $\nil$),
  if $\Betti(f)$ is a $\Ctwo$-equivariant  $v_{(n,m)}$\slfmp.
\end{defn}
\begin{rmk} The maps `multiplication by $2$' (of \autoref{egCtwo1}), $\hsf$ (of \autoref{egCtwo3}), and  $\eta_{1,1}$ (of \autoref{egCtwo2})   admit $\R$-motivic lifts along $\Betti$ and provide us with examples of a $v_{(0,0)}$\slfmp,  $v_{(0, \nil)}$\slfmp\  and $v_{(\nil, 0)}$\slfmp\  of the $\R$-motivic sphere spectrum $\bS_{\R}$, respectively. 
\end{rmk}

A theorem of Balmer and Sanders \cite{BaSa} asserts that $\cC( e, n) \subset  \cC( \Ctwo, m) $ if and only if $n \geq m+1$. 
In particular, $\cC(e,n)$ is contained in $\cC(\Ctwo, n-1)$. 
Consequently, there are no  type $(n, m)$ ($\Ctwo$-equivariant or $\R$-motivic) spectra if $n \geq m+2$. Their result also implies the following:

\begin{prop}
Let $X \in \ho(\Sp^{\mr{C}_2}_{2,\fin})$ be of type $(n+1,n)$ for some $n$. Then $X$ cannot support a $v_{(n+1,\nil)}$\slfmp.
\end{prop}

 The proposition holds since the cofiber of such a self-map would be of type $(n+2, n)$, contradicting the results of Balmer-Sanders.
  In particular, neither $\coneC{\hsf}$ nor $\coneR{\hsf}$ supports a $v_{(1,\nil)}$\slfmp. 
 However, it is possible that $\coneC{\hsf}$  as well as $\coneR{\hsf}$ can  admit a  $v_{(1, 0)}$\slfmp\ or a $v_{(\nil,0)}$\slfmp. In fact, $\eta_{1,1} \in \pi_{1,1}(\bS_\R)$ and $\eta_{1,1} \in \pi_{1,1}(\bS_{\Ctwo})$ induce  $v_{(\nil,0)}$\slfmp{s} of $\coneR{\hsf}$ and $\coneC{\hsf}$ respectively. 
In \autoref{nonexistence}, we show that:

\begin{thm} \label{thm:nonexist1}
 The spectrum $\coneR{\hsf}$  does not admit a $v_{(1,0)}$\slfmp. 
 \end{thm}
 
However, it is possible  that $\coneC{\hsf}$ admits a $v_{(1,0)}$\slfmp\ (for details see  \autoref{rem:Ctwov10}). In contrast to the classical case, there is no guarantee that a finite $\Ctwo$-equivariant  or $\R$-motivic spectrum will admit \emph{any} periodic self-map, or at least nothing concrete is known yet. This question must be studied!

The goal of this paper is rather modest. We consider the classical spectrum \[ \Y:= \M_2(1) \sma \mr{C}(\eta) \] that admits, up to homotopy, 8 different $v_1$\slfmp{s} of periodicity $1$ \cite[Section~2]{DM} (see also \cite{BEM}). We ask ourselves if the $v_1$\slfmp{s} can preserve symmetries upon providing  $\Y$ with interesting $\Ctwo$-equivariant structures. We will consider four $\Ctwo$-equivariant lifts of the spectrum $\Y$, 
\begin{enumerate}[(i)]
\item $\Ytriv^{\Ctwo}$, where the action of $\Ctwo$ is trivial, 
\item $\Yone^{\Ctwo} := \coneC{2} \sma \coneC{\eta_{1,1}} $, with $\Phi^{\Ctwo}(\Yone^{\Ctwo}) = \M_2(1) \sma \M_2(1)$,
\item $\Ytwo^{\Ctwo} := \coneC{\hsf} \sma \coneC{\eta_{1,0}} $, with $\Phi^{\Ctwo}(\Ytwo^{\Ctwo}) = \Sigma \mr{C}(\eta) \vee  \mr{C}(\eta)$,  and, 
\item $\Ythree^{\Ctwo} := \coneC{\hsf} \sma \coneC{\eta_{1,1}} $, with $\Phi^{\Ctwo}(\Ythree^{\Ctwo}) = \Sigma\M_2(1) \vee \M_2(1)$.
\end{enumerate} 
The $\Ctwo$-spectra $\Ytriv^{\Ctwo}$, $\Yone^{\Ctwo}$ and $\Ythree^{\Ctwo}$ are of type $(1,1)$, and $\Ytwo^{\Ctwo}$ is of type $(1,0)$.
There are unique $\R$-motivic lifts of the classes  $2$, $\hsf$, $\eta_{1,0}$,  and $\eta_{1,1}$, and therefore we have unique $\R$-motivic lifts of  $\Ytriv^{\Ctwo}$, $\Yone^{\Ctwo}$, $\Ytwo^{\Ctwo}$, and $\Ythree^{\Ctwo}$ which we will simply denote by $\Ytriv^\R$, $\Yone^\R$, $\Ytwo^\R$, and $\Ythree^{ \R}$, respectively. In this paper we prove:

\begin{thm} \label{main1} The $\R$-motivic spectrum $\Ythree^\R$ admits a  $v_{1, \mr{nil}}$\slfmp 
\[
\begin{tikzcd}
 v:\Sigma^{2,1}\Ythree^\R \rar & \Ythree^\R
 \end{tikzcd}
 \]
 of periodicity $1$. 
\end{thm}

By applying the Betti realization functor we get: 

\begin{cor} 
The $\Ctwo$-equivariant spectrum $\Ythree^{\Ctwo}$ admits a $1$-periodic  $v_{1, \mr{nil}}$\slfmp 
\[
\begin{tikzcd}
 \Betti(v):\Sigma^{2,1}\Ythree^{\Ctwo} \rar & \Ythree^{\Ctwo}.
 \end{tikzcd}
 \]
\end{cor}
 \begin{cor} 
 The geometric fixed-point spectrum of the telescope   \[ \Betti(v)^{-1}\Ythree^{\Ctwo} \] is contractible. 
 \end{cor}
 
Classically, the cofiber of a $v_1$\slfmp\ on $\Y$ is a realization of the finite subalgebra $\cA(1)$ of the Steenrod algebra $\cA$. We see a very similar phenomenon in the $\R$-motivic as well as in the $\Ctwo$-equivariant cases. The $\Ctwo$-equivariant Steenrod algebra $\cAC$ as well as the $\R$-motivic Steenrod algebra $\cAR$ admit subalgebras analogous to $\cA(1)$ (generated by $\Sq^1$ and $\Sq^2$) \cites{hill,ricka}, which we denote by $\cA^{\Ctwo}(1)$ and $\cA^{\R}(1)$, respectively. We observe that: 
 
 \begin{thm}\label{thm:cofibv1} 
 The spectrum $ \coneR{v}:= \mr{Cof} (v:\Sigma^{2,1}\Ythree^{\R} \to \Ythree^{\R})$  is a type $(2,1)$ complex whose bigraded cohomology is a free $\cAR(1)$-module on one generator.  
 \end{thm}
 \begin{cor} The bigraded cohomology of the $\Ctwo$-equivariant spectrum $$\coneC{\Betti(v)}\cong \Betti(\coneR{v})$$ is a free $\cAC(1)$-module on one generator.
 \end{cor}
 
 Our last main result in this paper is the following. 
\begin{thm} \label{thm:nonexist2}
 The spectrum $\Ytwo^{\R}$  does not admit a $v_{(1,0)}$\slfmp. 
 \end{thm}

 The above results  immediately raise some obvious questions. Pertaining to self-maps one may ask: 
 Does $\Yone^\R$ admit a $v_{1, \nil}$\slfmp?
  Does $\Yone^\R$ or $\Ythree^\R$  admit a $v_{(1,1)}$\slfmp?
   Does $\Ytriv^\R$, $\Yone^\R$ or $\Ythree^\R$ admit $v_{(\nil,1)}$\slfmp? 
 Or more generally, how many different homotopy types of each kind of periodic self-maps exist? Related to $\cA^\R(1)$, one may inquire: How many different $\cA^\R$-module structures can be given to $\cA^\R(1)$? 
  Can those $\cA^\R$-modules be realized as a spectrum? Are the realizations of $\cA^\R(1)$  equivalent to cofibers of periodic self-maps of $\Y^\R_{(i,j)}$?  We hope to address most, if not all, of the above questions in our upcoming work (see  \autoref{rem:A128}, \autoref{rem:Y1v1} and \autoref{rem:Ctwov10}). 
 
 \subsection{Outline of our method}
\label{outline}
 We first construct a spectrum $\cA_1^\R$ which realizes the algebra $\cA^{\R}(1)$ using a method of Smith (outlined in \cite[Appendix~C]{Rav}) which constructs new finite spectra (potentially with larger number of cells) from known ones. The idea is as follows. If $X$ is a $p$-local finite spectrum then the permutation group $\Sigma_n$ acts on $X^{\sma n}$. One may then use an idempotent $e \in \Z_{(p)}[\Sigma_n]$ to obtain a split summand of the spectrum $X^{\sma n}$. As explained in \cite[Appendix~C]{Rav}, Young tableaux provide a rich source of such idempotents. For a judicious choice of $e$ and $X$, the spectrum $e(X^{\sma n})$ can be interesting. 
 
We exploit the relation that $\hsf \cdot \eta_{1,1} = 0$ 
in $\pi_{\ast, \ast}(\bS_\R)$ \cite{Morel} to construct an $\R$-motivic analogue of  the question mark complex. The cell-diagram of the question mark complex is as described in the picture below. 
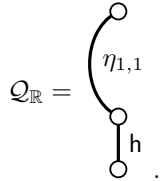
\begin{figure}[h]
\[
\QR = \!\! \raisebox{-3em}{
\begin{tikzpicture}\begin{scope}[very thick, every node/.style={sloped,allow upside down}, scale=0.7]

\draw (0,0)  node[inner sep=0] (v0) {} -- (0,1) node[inner sep=0] (v1) {} ;
\draw (0,3)  node[inner sep=0] (v3) {};
 \draw(v1) to [out=150,in=-150] (v3);
\filldraw (v0) circle (4pt);
\filldraw (v1) circle (4pt);
\filldraw (v3) circle (4pt);
\filldraw[white] (v0) circle (3pt);
\filldraw[white] (v1) circle (3pt);
\filldraw[white] (v3) circle (3pt);
\draw (0,.5) node[right]{$\hsf$} (-.5,2) node[right]{$\eta_{1,1}$};
\end{scope}\end{tikzpicture}. }
\]
\caption{ Cell-diagram of the $\R$-motivic question mark complex}
\end{figure}
For a choice of idempotent element $e$ in the group ring  $\Z_{(2)}[\Sigma_3]$,
we observe that $e(\rH^{*,*}(\QR)^{\otimes 3})$ is a free $\cA^\R(1)$-module.
This is the cohomology of an $\R$-motivic spectrum $\tilde{e}(\QR^{\sma 3})$, which we call $\Sigma^{1,0}\cA_1^\R$ (see \eqref{defnA1} for details). The observation requires us to develop a criterion that will detect freeness for modules over certain subalgebras of $\cA^\R$.   Writing $\bMR$ for the $\R$-motivic cohomology of a point, 
we prove: 

 \begin{thm} \label{thm:freeR} 
 A finitely generated
$\cA^\R(1)$-module  $\M$ is free if and only if  
\begin{enumerate}
\item $\M$ is free as an $\bMR$-module, and
\item $ \F \otimes_{\bMR} \M$ is a free $\F \otimes_{\bMR} \cA^{\R}(1)$-module.
\end{enumerate}
\end{thm}

The cohomology  of $\cA_1^\R$ provides an $\cA^\R$-module structure on $\cA^\R(1)$, which immediately gives us a short exact sequence 
\[ 0 \to \mr{H}^{*,*}(\Sigma^{3,1}\Ythree^\R)  \to \mr{H}^{*,*}(\cA_1^\R) \to \mr{H}^{*,*}(\Ythree^\R) \to 0\]
 of  $\cA^\R$-modules. Thus, we get a candidate for a $v_{1, \nil}$\slfmp\ in the $\R$-motivic Adams spectral sequence 
\[ \overline{v} \in \Ext_{\cA^\R}^{*,*,*}(  \mr{H}^{*,*}(\Ythree^\R),  \mr{H}^{*,*}(\Ythree^\R)) \Rightarrow [\Ythree^\R, \Ythree^\R]_{*,*}\]
which survives as there is no potential target for a differential supported by $\overline{v}$.

 \subsection*{Organization of the paper}
 In \autoref{sec:free}, we review the $\R$-motivic Steenrod algebra $\cA^\R$,  discuss the structure of its subalgebra $\cA^\R(n)$, and  prove \autoref{thm:freeR}. 
In \autoref{sec:Realization}, we construct the spectrum $\cA_1^\R$ that realizes the subalgebra $\cA^\R(1)$ and prove that it is of type $(2,1)$. 
In \autoref{selfmap}, we prove \autoref{main1} and \autoref{thm:cofibv1}; i.e., we show that $\Ythree^\R$ admits a $v_{1,\nil}$\slfmp\ and that its cofiber has the same $\cA^\R$-module structure as that of $\mr{H}^{*,*}(\cA_1^\R)$. 
In \autoref{nonexistence}, we show the non-existence of a $v_{(1,0)}$\slfmp\ on $\coneR{\hsf}$ and $\Ytwo^\R$; i.e., we prove \autoref{thm:nonexist1} and \autoref{thm:nonexist2}. 
 
\subsection*{ Acknowledgement}
The authors are indebted to Nick Kuhn for  explaining some of the subtle points of Smith's work exposed in \cite{Rav}[Appendix C], which is the key idea behind \autoref{main1}. The authors also benefited from  conversations with Mark Behrens,  Dan Isaksen, and Zhouli Xu.

\section{The $\R$-motivic Steenrod algebra and a freeness criterion} \label{sec:free}

We begin by reviewing the $\R$-motivic Steenrod algebra $\cAR$ following Voevodsky \cite{V}.
 The algebra $\cAR$ is the bigraded homotopy classes of  self-maps of the $\R$-motivic Eilenberg-Mac~Lane spectrum $\mr{H}\F^\R$:
\[ \cAR = [\mr{H}\F^\R, \mr{H}\F^\R ]_{\ast, \ast}.  \]
The unit map $\bS_\R \to \mr{H}\F^\R$ induces a canonical projection map 
\[
\epsilon: \cA^{\R} \longrightarrow  \bMR :=  [\bS_\R, \mr{H}\F^\R]_{*,*} \iso \F[\tau, \rho], 
\]
where $|\tau| = (0,-1)$ and $|\rho| = (-1,-1)$. 
Further, using the multiplication map $\mr{H}\F^\R \sma \mr{H}\F^\R \to \mr{H}\F^\R$ one can give $\cAR$ a left $\bMR$-module structure as well as a right $\bMR$-module structure. Voevodsky shows that $\cAR$ is a free left $\bMR$-module. There is an analogue of the classical Adem basis in the motivic setting, and Voevodsy established motivic Adem relations, thereby completely describing the multiplicative structure of $\cAR$. 

The motivic Steenrod algebra $\cAR$ also admits a diagonal map, so that its left $\bMR$-linear dual is an algebra over $\F$. Note that $\cAR$ is an $\F$-algebra but not an $\bMR$-algebra as $\tau$ is not a central element 
since
 \begin{equation} \label{taunotcentral}
 \Sq^1(\tau) = \rho \neq \tau \Sq^1. 
 \end{equation}
This complication is also reflected in the fact that  the pair $(\bMR, \hom_{\bMR}(\cAR,\bMR))$
is a Hopf-algebroid, and not a Hopf-algebra like its complex counterpart. The underlying algebra of the dual $\R$-motivic Steenrod algebra is given by 
\[ \cA^{\R}_{\ast} \cong \mathbb{M}_2^\R[\xi_{i+1}, \tau_i: i \geq 0 ]/( \tau_i^2 = \tau \xi_{i+1} + \rho \tau_{i+1} + \rho \tau_0 \xi_{i+1} ) \]
where $\xi_{i}$ and $\tau_{i}$ live in bidegree $(2^{i+1} -2, 2^i -1)$ and $(2^{i+1} -1, 2^i -1)$, respectively. 
The complete description of the Hopf-algebroid structure can be found in \cite{V}.

 Ricka\footnote{
Ricka actually identified the quotient Hopf-algebroids of the $\Ctwo$-equivariant dual Steenrod algebra. However, the difference between  the $\R$-motivic Steenrod algebra and the $\Ctwo$-equivariant Steenrod algebra lies only in the coefficient rings and results of Ricka easily  identifies the quotient Hopf-algebroids of the $\R$-motivic Steenrod algebra. 
} \cite{ricka} identified the quotient Hopf-algebroids of $\cA^\R_\ast$ (see also \cite{hill}).
In particular, there are quotient Hopf-algebroids 
\[ \cA^\R(n)_\ast= \cA^\R_{\ast}/ ( \xi_1^{2^{n}}, \dots , \xi_n^2, \xi_{n+1}, \dots, \tau_0^{2^{n+1}}, \dots,\tau_{n}^2,   \tau_{n+1}, \dots   )  \]
which can be thought of as analogues of the quotient Hopf-algebra 
\[ \cA(n)_\ast = \cA_\ast / (\xi_1^{2^{n+1}}, \dots, \xi_{n+1}^{2}, \xi_{n+2}, \dots  )  \]  
of the classical dual Steenrod algebra $\cA_*$. 
It is an algebraic fact that 
\[ \tau^{-1}(\cA^\R(n)_\ast /(\rho)) \cong \F[\tau^{\pm 1}] \otimes \cA(n)_\ast \]
as Hopf algebras. The  quotient Hopf-algebroid $\cA^\R(n)_\ast$ is the $\mathbb{M}_2^\R$-linear dual of the subalgebra $\cA^\R(n)$ of $\cA^\R$, which is generated  by $\{ \tau, \rho, \Sq^1, \Sq^2, \dots, \Sq^{2^{n}} \}$.  

Even though $\tau$ is not in the center \eqref{taunotcentral}, $\rho$ is in the center. We make use of this fact to prove the following result. 

\begin{lem} \label{lem:freestep1}
A finitely-generated
 $\cA^\R(n)$-module  $\M$ is free if and only if  
\begin{enumerate}
\item $\M$ is free as an $\F[\rho]$-module, and,
\item $\M/(\rho)$ is a free $\cA^\R(n)/(\rho)$-module.
\end{enumerate}
\end{lem}

\begin{proof} The `only if' part is trivial. For the `if' part, choose a basis $\mathcal{B} = \{ b_1, \dots, b_n \}$ of $\M/(\rho)$ and let $\tilde{b}_i\in M$ be any lift of $b_i$. Let $\mr{F}$ denote the free $\cA^\R(n)$-module generated by $\mathcal{B}$ and consider the map
\[ f: \mr{F}  \to \M  \]
which sends $b_i \mapsto \tilde{b}_i$. We show that $f$  is an isomorphism  by inductively proving that  $f$ induces an isomorphism $\mr{F}/(\rho^n) \cong \mr{M}/(\rho^n)$  for all $n \geq 1$. The case of $n=1$ is true by assumption.

For the inductive argument,  first note that the diagram 
\[ 
\begin{tikzcd}
0 \rar \dar[equals] & \mr{F}/(\rho^{n-1}) \rar["\cdot \rho"] \dar["f_{n-1}"] & \mr{F}/(\rho^n) \rar \dar[ "f_n"] & \mr{F}/(\rho) \dar["f_0"] \rar  & 0 \dar[equals]  \\
0 \rar & \M/(\rho^{n-1}) \rar["\cdot \rho"] & \M/(\rho^n) \rar[""] & \M/(\rho) \rar& 0
\end{tikzcd}
\] 
is a diagram of $\cA^\R(n)$-modules (since $\rho$ is in the center) where the horizontal rows are exact. The map $f_0$ is an isomorphism by assumption $\emph{(2)}$, and $f_{n-1}$ is an isomorphism by the inductive hypothesis; hence, $f_n$ is an isomorphism by the five lemma. 
\end{proof}

\begin{proof}[Proof of \autoref{thm:freeR}] The result follows immediately from \autoref{lem:freestep1} combined with \cite[Theorem B]{HeKr} and the fact that 
\[ \cA^\C(n) = \cA^\R(n)/(\rho). \qedhere \]
\end{proof}

The  work of Adams and Margolis \cite{AM} provides a freeness criterion for modules over finite-dimensional subalgebras of the classical Steenrod algebra. For an $\cA(n)$-module $\mr{M}$ and  element $x \in \cA(n) $ such that $x^2 = 0$, one can define the Margolis homology of $\mr{M}$ with respect to $x$ as 
\[ \mathcal{M}(\mr{M}, x) = \frac{ \ker(x: \mr{M} \to \mr{M} )}{\mr{img}( x: \mr{M} \to \mr{M}) } .\]

\begin{thm}\cite[Theorem~4.4]{AM} \label{thm:AdamsMargolis}  A finitely generated  $\cA(n)$-module $\mr{M}$ is free if and only if  $\mathcal{M}(\mr{M}, \mr{P}^s_t) = 0$ for $0 <s < t$ with $s+t \leq n$. 
\end{thm}

\begin{rmk} In the classical Steenrod algebra,  $\mr{P}^s_t$ is the element dual to $\xi_t^{2^s}$. In terms of the 
Milnor basis, this is
  $\Sq(\overbrace{0, \dots, 0}_{t-1}, 2^s)$. The element $\mr{P}^0_t$ is often denoted by $\mr{Q}_{t-1}$. 
\end{rmk}

Note that

\[ \cA^\R(n)_\ast / (\rho, \tau) = \frac{\F[\xi_1, \dots, \xi_{n}]}{(\xi_1^{2^{n} }, \dots, \xi_{n}^2 )} \otimes \Lambda(\tau_0, \dots, \tau_{n})  \] 
as a Hopf-algebra. Further, if we forget the motivic grading, we have an isomorphism 
\begin{equation} \label{quotientidentify}
\cA^\R(n) / (\rho, \tau) \cong \varphi \cA(n-1) \otimes \Lambda(\mr{P}_1^0, \dots, \mr{P}_{n}^0),
\end{equation}
where $\varphi \cA(n-1)$ denotes the `double' (see \cite[Chapter~15, Proposition~11]{BookMargolis}) of  $\cA(n-1)$. Let $$\overline{\mr{P}}^s_t = (\xi^{2^s}_t)^* \in \cA^\R(n) .$$ It can be shown that 
\[ (\overline{\mr{P}}^s_t)^2 \equiv  0 \mod (\rho, \tau) \]
 for $s \leq t$.  Combining \eqref{quotientidentify}, \autoref{thm:AdamsMargolis} and a similar result for primitively generated exterior Hopf-algebras \cite[Theorem~2.2]{AM},  we deduce: 
\begin{lem} \label{modulofree} A finitely generated $\cA^\R(n) / (\rho, \tau)$-module  $\overline{\mr{M}}$ is free if and only if $\mathcal{M}(\overline{\mr{M}}, \overline{\mr{P}}^s_t) = 0$ whenever $0 \leq s\leq t$ and $1\leq s + t \leq n+1$.
\end{lem}
 We end this section by recording the following corollary, which is immediate from \autoref{thm:freeR} and \autoref{modulofree}.
\begin{cor} \label{cor:freeR} A finitely generated $\cA^\R(n)$ module  $\M$ is free if and only if  
\begin{enumerate}
\item $\M$ is free as an $\mathbb{M}_2^\R$-module, and, 
\item $ \mathcal{M}( \M \otimes_{\mathbb{M}_2^{\R}} \F, \overline{\mr{P}}_s^t) = 0$ for $0 \leq t \leq s$ and $s + t = n+1$.
\end{enumerate}
\end{cor}
\section{A realization of $\cAR(1)$}
\label{sec:Realization}
Consider the $\R$-motivic question mark complex $\QR$, as introduced in Subsection~\ref{outline}. Let $\Sigma_n$ act on $\QR^{\sma n}$ by permutation. Any element $e \in  \Z_{(2)}[\Sigma_n]$ produces a canonical map 
\[ 
\begin{tikzcd}
\tilde{e}: \QR^{\sma n} \rar & \QR^{\sma n}.
\end{tikzcd}
\]
Now let $e$ be the idempotent
 \[ \mbox{$e= \frac{1+(1\ 2) - (1 \ 3) - (1\ 3 \ 2)}3$} \] 
  in $\Z_{(2)}[\Sigma_3]$,  and denote by $\overline{e}$ the resulting idempotent of $\F[\Sigma_3]$. We record the following important property of $\overline{e}$ which is a special case of  \cite[Theorem C.1.5]{Rav}.

  \begin{lem} \label{lem:vanish} If $V$ is a finite-dimensional $\F$-vector space, then $\overline{e}(V^{\otimes 3})=0$ if and only if $\dim V\leq 1$. 
  \end{lem}

The following result, which gives the values of $\overline{e}$ on induced representations, is also straightforward to verify:

\begin{lem}\label{eInduced} 
Suppose that $W= \mathrm{Ind}_{C_2}^{\Sigma_3} \F$ is induced up from the trivial representation of a cyclic 2-subgroup. Then $\overline{e}(W) \iso \F$. Moreover, for the regular representation $\F[\Sigma_3] = \mathrm{Ind}_{e}^{\Sigma_3} \F$, we have 
$\dim \overline{e}( \F[\Sigma_3]) = 2$.
\end{lem}

  We also record the fact that  when  $\dim_{\F} V = 2$ and $\dim_{\F} W = 3$ then 
  \begin{equation} \label{dimcount}
  \dim_{\F} \overline{e}(V^{\otimes 3}) = 2 \quad \text{and} \quad
  \dim_{\F} \overline{e}(W^{\otimes 3}) = 8,
    \end{equation}
  as we will often use this.   

 The bottom cell of $\tilde{e}(\QR^{\sma 3})$ is in degree $(1,0)$, and we
 define 
\begin{equation} \label{defnA1}
 \cA_1^\R := \Sigma^{-1,0}\tilde{e}(\QR^{\sma 3}) = \Sigma^{-1,0}\underset{\to}{\mr{hocolim}}(  \QR^{\sma 3} \overset{\tilde{e}}\to  \QR^{\sma 3} \overset{\tilde{e}}\to \dots  ).
 \end{equation}
The purpose of this section is to  prove the following theorem.

\begin{thm}
\label{thm:A1R}
The spectrum $\cA_1^\R$ is a type $(2,1)$ complex whose bi-graded cohomology $\mr{H}^{*,*}(\cA_1^\R)$ is a free $\cA^\R(1)$-module on one generator.
\end{thm}

\subsection{$\cAR_1$ is of type $(2,1)$}
 Let $ \cA_1^{\Ctwo} := \Betti(\cA_1^\R) \text{ and } \QC := \Betti(\QR).$ Note that we have a $\Ctwo$-equivariant splitting 
\[ \QC^{\sma 3} \simeq \tilde{e}(\QC^{\sma 3}) \vee (1 - \tilde{e})(\QC^{\sma 3})\]
which splits the underlying spectra as well as  the geometric fixed-points, as both $\Phi^e$ and $\Phi^{\Ctwo}$ are additive functors.

 We will identify the underlying spectrum $\Phi^e(\cA_1^{\Ctwo})$ by studying the $\cA$-module structure of its cohomology with $\F$-coefficients. 
Firstly, note that
\[ \Phi^e(\cA_1^{\Ctwo}) \simeq \Sigma^{-1} \tilde{e}( \Phi^e(\QC^{\sma 3}) ) \simeq   \Sigma^{-1}\tilde{e}( \cQ^{\sma 3} ), \]
where $\cQ$ is the classical question mark complex, whose $\mr{H}\F$-cohomology as an $\cA$-module is well understood. It consists of three $\F$-generators $a$, $b$, and $c$ in internal degrees $0$, $1$, and $3$, such that $\Sq^1(a) = b$ and $\Sq^2(b) = c$ are the only nontrivial relations, as displayed in \autoref{fig:HQ}.
\begin{figure}[h]
\[
\mr{H}^*(\cQ; \F)  = \!\! \raisebox{-3em}{
\begin{tikzpicture}\begin{scope}[ thick, every node/.style={sloped,allow upside down}, scale=0.7]
\draw (0,0)  node[inner sep=0] (v0) {} -- (0,1) node[inner sep=0] (v1) {};
\draw (0,3) node[inner sep=0] (v3) {};
 \draw[blue] (v1) to [out=150,in=-150] (v3);
\filldraw (v0) circle (2.5pt);
\filldraw (v1) circle (2.5pt);
\filldraw (v3) circle (2.5pt);
\draw (0,0) node[right]{$\ a$};
\draw (0,1) node[right]{$\ b$};
\draw (0,3) node[right]{$\ c$};
\end{scope}\end{tikzpicture}}
\]
\caption{We depict the $\cA$-structure of $\mr{H}^{*}(\cQ; \F)$ by marking $\Sq^1$-action by black straight lines and $\Sq^2$-action by blue curved lines between the $\F$-generators. }
\label{fig:HQ}
\end{figure}
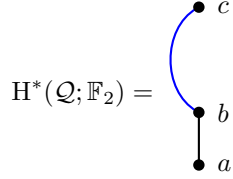

 Because of the Kunneth isomorphism and the fact that the  Steenrod algebra is cocommutative, we have an isomorphism of $\cA$-modules 
\[ \mr{H}^{*+1}( \Phi^e(\cA_1^{\Ctwo});\F) \cong \mr{H}^*(\tilde{e}(\cQ^{\sma 3}); \F) \iso  \overline{e}(\mr{H}^*(\cQ; \F)^{\otimes 3}). \]

\begin{lem} \label{lem:underA} The underlying  $\cA(1)$-module structure of $\mr{H}^*( \Phi^e(\cA_1^{\Ctwo});\F)$ is free on a single generator. 
\end{lem}

\begin{proof} Let us denote the $\cA$-module $\mr{H}^*(\cQ; \F)$ by $\mr{V}$. Since $\dim \mathcal{M}(\mr{V}, Q_i) = 1$ for $i \in \{ 0,1 \}$,  it follows from the Kunnneth isomorphism of  $Q_i$-Margolis homology groups, cocommutativity of the Steenrod algebra, and \autoref{lem:vanish} that  
\[ \mathcal{M}(\overline{e}(\mr{V}^{\otimes 3}), Q_i) =  \overline{e}(\mathcal{M}(\mr{V}, Q_i)^{\otimes 3}) =0 \]
for $i=\{ 1,2 \}$. It follows from \cite[Theorem~3.1]{AM}
that $\mr{H}^*( \Phi^e(\cA_1^\R);\F)$ is free as an $\cA(1)$-module. It is singly generated because of \eqref{dimcount}. 
\end{proof}
We explicitly identify the image  of 
$ \overline{e}: \mr{H}^*(\cQ; \F)^{\otimes 3} \longrightarrow \mr{H}^*(\cQ; \F)^{\otimes 3} $ in \autoref{fig:A1}. 
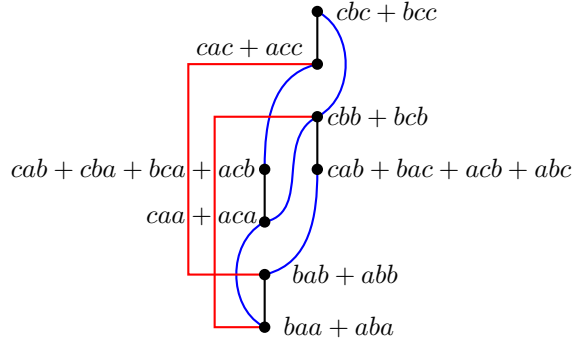
\begin{figure}[h]
\[
\begin{tikzpicture}\begin{scope}[thick, every node/.style={sloped,allow upside down}, scale=0.7]
\draw (0,0)  node[inner sep=0] (v00) {} -- (0,1) node[inner sep=0] (v01) {};
\draw (0,2)  node[inner sep=0] (v11) {} -- (0,3) node[inner sep=0] (v12) {};
\draw (1,3)  node[inner sep=0] (v22) {} -- (1,4) node[inner sep=0] (v23) {};
\draw (1,5)  node[inner sep=0] (v33) {} -- (1,6) node[inner sep=0] (v34) {};
 \draw [color=blue] (v00) to [out=150,in=-150] (v11);
 \draw [color=blue] (v01) to [out=15,in=-90] (v22);
 \draw [color=blue] (v12) to [out=90,in=-165] (v33);
 \draw [color=blue] (v23) to [out=30,in=-30] (v34);
 \draw [color=blue] (v11) to [out=15,in=-150] (v23);
 \draw [color=red] (0,0) -- (-.95,0) -- (-.95,4) -- (1,4);
 \draw [color=red] (0,1) -- (-1.45,1) -- (-1.45,5) -- (1,5);
\filldraw (v00) circle (2.5pt);
\filldraw (v01) circle (2.5pt);
\filldraw (v11) circle (2.5pt);
\filldraw (v12) circle (2.5pt);
\filldraw (v22) circle (2.5pt);
\filldraw (v23) circle (2.5pt);
\filldraw (v33) circle (2.5pt);
\filldraw (v34) circle (2.5pt);
\draw (0,0) node[right]{$\ baa+aba$} (0,1) node[right]{$\ \ bab+abb$} (0,2) node[left,xshift={1ex},yshift={2pt}]{$caa+aca \ $ } (0,3) node[left]{$cab+cba+bca+acb$};
\draw (1,6) node[right]{$\ cbc+bcc$} (1.5,5.3) node[left,xshift={-2ex}]{$cac+acc \ $} (1,4) node[right]{$  cbb+bcb $} (1,3) node[right]{$cab+bac+acb+abc$};
\end{scope}\end{tikzpicture}
\]
\caption{The $\cA$-module structure of $\mr{H}^{*}(\Phi^e(\cA_1^{\Ctwo}); \F)$ :  Black straight lines,  blue curved lines, and   red boxed lines represent the $\Sq^1$-action, $\Sq^2$-action, and $\Sq^4$-action, respectively.} 
\label{fig:A1}
\end{figure}

\begin{rmk}  \label{rem:A00}Using the Cartan formula, we can identify the action of $\Sq^4$ on $\Phi^e(\cA_1^{\Ctwo})$. We notice that its $\cA$-module structure is isomorphic to $A_1[10]$ of \cite{BEM}. Since such an $\cA$-module is realized by a unique $2$-local finite spectrum, we conclude \[ \Phi^e(\cA_1^{\Ctwo}) \simeq A_1[10] \]
and is of type $2$. 
\end{rmk}

Our next goal is to understand the homotopy type of the geometric fixed-point spectrum $\Phi^{\Ctwo}(\cA_1^{\Ctwo})$. First observe that the geometric fixed-points of the $\Ctwo$-equivariant question mark
 complex $\QC$ is the \emph{exclamation mark} complex 
\[ 
\mathcal{E} := \ \ \!\! \raisebox{-2em}{
 \begin{tikzpicture}\begin{scope}[thick, every node/.style={sloped,allow upside down}, scale=0.7]
\draw (0,0)  node[inner sep=0] (v0) {};
\draw (0,1) node[inner sep=0] (v1) {};
\draw (0,2) node[inner sep=0] (v3) {};
 \draw (v1) to  (v3);
\filldraw (v0) circle (4pt);
\filldraw (v1) circle (4pt);
\filldraw (v3) circle (4pt);
\filldraw[white] (v0) circle (3pt);
\filldraw[white] (v1) circle (3pt);
\filldraw[white] (v3) circle (3pt);

\end{scope}\end{tikzpicture}} \ \  \simeq \ \ \bS^0 \vee \Sigma\M_2(1)!
\]
This is because  $\Phi^{\Ctwo}(\hsf) = 0 $ and $\Phi^{\Ctwo}(\eta_{1,1}) = 2$. Secondly,
\[  
\mr{H}^{*+1}( \Phi^{\Ctwo}(\cA_1^{\Ctwo}); \F  ) \iso 
\mr{H}^* ( \tilde{e}( \mathcal{E}^{\sma 3} ); \F) \iso
 \overline{e}( \mr{H}^*( \mathcal{E} ; \F  )^{\otimes 3})
\]
is an isomorphism of $\cA$-modules. 
\begin{figure}[h] 
\[
\begin{tikzpicture}\begin{scope}[ thick, every node/.style={sloped,allow upside down}, scale=0.7]
\draw (0,0)  node[inner sep=0] (v00) {} -- (0,1) node[inner sep=0] (v01) {};
\draw (1,1)  node[inner sep=0] (v11) {} -- (1,2) node[inner sep=0] (v12) {};
\draw (2,2)  node[inner sep=0] (v22) {} -- (2,3) node[inner sep=0] (v23) {};
\draw (3,3)  node[inner sep=0] (v33) {} -- (3,4) node[inner sep=0] (v34) {};
 \draw [color=blue] (v11) to [out=15,in=-150] (v23);
\filldraw (v00) circle (2.5pt);
\filldraw (v01) circle (2.5pt);
\filldraw (v11) circle (2.5pt);
\filldraw (v12) circle (2.5pt);
\filldraw (v22) circle (2.5pt);
\filldraw (v23) circle (2.5pt);
\filldraw (v33) circle (2.5pt);
\filldraw (v34) circle (2.5pt);
\draw (0,0) node[left]{$yxx+xyx$} (0,1) node[left]{$zxx+xzx$} (1,1) node[right]{$xyy+yxy$} (1,2) node[left]{$zxy+xzy+yxz+xyz$};
\draw (3,4) node[right]{$zyz+yzz$} (3,3) node[right]{$zyy+yzy$} (2,3) node[left]{$xzz+zxz$} (2,2) node[right]{$xzy+zxy+zyx+yzx$};
\end{scope}\end{tikzpicture}
\] 
\caption{The $\cA$-module structure of $\mr{H}^*( \Phi^{\Ctwo}(\cA_1^{\Ctwo}); \F  )$. }
\label{figure:gfixA1}
\end{figure}
We explicitly calculate the $\cA$-module structure of  $\mr{H}^*( \Phi^{\Ctwo}(\cA_1^{\Ctwo}); \F  )$ from the above isomorphism and record it in \autoref{figure:gfixA1}  as a subcomplex of $\mr{H}^*(\mathcal{E}; \F)^{\otimes 3}$, with the convention that $x$, $y$ and $z$ are generators in $\mr{H}^*(\mathcal{E}; \F)$ in degree $0$, $1$ and $2$ respectively. 

\begin{lem} \label{lem:splitfixA} The finite spectrum $\Phi^{\Ctwo}(\cA_1^{\Ctwo}) $  is a  type $1$ spectrum and equivalent to 
\[ \Phi^{\Ctwo}(\cA_1^{\Ctwo}) \simeq \M_2(1) \vee \Sigma \Big(\M_2(1) \sma \M_2(1)\Big) \vee \Sigma^3 \M_2(1).\] \end{lem}

\begin{proof} 
From \autoref{figure:gfixA1}, it is clear that  
we have an isomorphism of $\cA$-modules
\[
\mr{H}^*(\Phi^{\Ctwo}(\cA_1^{\Ctwo}); \F) \cong \mr{H}^*\Big(\M_2(1) \vee \Sigma \big(\M_2(1) \sma \M_2(1)\big) \vee \Sigma^3 \M_2(1); \F\Big).
\]
  It is possible that the $\cA$-module $\mr{H}^*(\Phi^{\Ctwo}(\cA_1^{\Ctwo}); \F)$ may not realize  to a unique finite spectrum (up to weak equivalence). However, other  possibilities can be eliminated from the fact that $\mathcal{E}^{\sma 3}$
  splits $\Sigma_3$-equivariantly into four components: 
\[ \mathcal{E}^{\sma 3} \simeq  \bS \vee \left(\bigvee_{i = 1}^3\Sigma \M_2(1)\right) \vee \left(\bigvee_{i = 1}^3\Sigma^2 \M_2(1)^{\sma 2} \right) \vee \Sigma^3 \M_2(1)^{\sma 3} . \]
The idempotent $\tilde{e}$ annihilates $\bS \iso \bS^{\sma 3}$, and \autoref{eInduced} implies that
\[ \tilde{e} \left( \bigvee_{i=1}^3 \Sigma \M_2(1) \right) \simeq \Sigma \M_2(1) \qquad \text{and}\] 
\[ \tilde{e} \left( \bigvee_{i=1}^3 \Sigma^2 \M_2(1)\sma \M_2(1) \right) \simeq \Sigma^2 \M_2(1) \sma \M_2(1).
\]
Similarly, we see using \eqref{dimcount} that 
\[\mr{H}^*\left( \tilde{e} \left( \Sigma^3 \M_2(1)^{\sma 3} \right) \right) \iso \overline{e} \left( \mr{H}^*\left( \Sigma \M_2(1) \right)^{\otimes 3}\right) \iso \mr{H}^*(\Sigma^3 \M_2(1)).\]
Hence, the result. 
\end{proof}

\subsection{The cohomology of $\cAR_1$ is free over $\cAR(1)$}
 Next, we analyze the $\cA^\R$-module structure of $\mr{H}^{*,*}(\cA_1^\R)$. We  begin by recalling some general properties of the cohomology of motivic spectra.  

 If $X, Y \in \Sp^\R_{2,\mr{fin}}$ such that  $\mr{H}^{\ast, \ast}(X)$ is free as a left  $\mathbb{M}_2^\R$-module, then we have a Kunneth isomorphism \cite[Proposition 7.7]{DI}
 \begin{equation} \label{eqn:Kunnethiso}
 \mr{H}^{\ast, \ast}(X \sma Y ) \cong \mr{H}^{*,*}(X) \otimes_{\mathbb{M}_2^\R} \mr{H}^{*,*}(Y) 
 \end{equation}
 as the relevant Kunneth spectral sequence  collapses. Further, if $\mr{H}^{\ast, \ast}(Y)$ is free as a left $\mathbb{M}_2^\R$-module, then so is  $\mr{H}^{\ast, \ast}(X \sma Y )$. 
 The $\cA^\R$-module structure of  $\mr{H}^{\ast, \ast}(X \sma Y )$ can then be computed using the Cartan formula. 
The comultiplication map  of $\cA^\R$ is  left $\mathbb{M}_2^\R$-linear,  coassociative and cocommutative \cite[Lemma 11.9]{V}, which is also reflected in the fact that its $\mathbb{M}_2^\R$-linear dual is a commutative and associative algebra. Thus, when $\mr{H}^{\ast, \ast}(X)$ is a free  left  $\mathbb{M}_2^\R$-module, the elements of $\F[\Sigma_n]$ acts on 
 \[ \mr{H}^{\ast, \ast}(X^{\sma n}) \cong \mr{H}^{\ast, \ast}(X) \otimes_{\mathbb{M}_2^\R} \dots \otimes_{\mathbb{M}_2^\R}  \mr{H}^{\ast, \ast}(X)\] 
 via permutation and commutes with the action of $\cA^{\R}$. This also implies that $\F[\Sigma_n]$ also acts on 
 \[ \mr{H}^{\ast, \ast}(X^{\sma n})/(\rho, \tau) \cong (\mr{H}^{\ast, \ast}(X)/(\rho, \tau)) \otimes \dots \otimes \mr{H}^{\ast, \ast}(X)/(\rho, \tau)  \]
 and commutes with the action of $\cA^\R \mm \mathbb{M}_2^\R$. From the above discussion we may conclude that  
 \begin{equation} \label{eqn:identifycohomology}
  \mr{H}^{*, *}( \cA_1^\R) \cong \Sigma^{-1} \overline{e}(\mr{H}^{*, *}(\QR)^{\otimes 3} ) 
  \end{equation}
 is an isomorphism of $\cA^\R$-module. 
 
 We will also rely upon the following important property of the action of the motivic Steenrod algebra on the cohomology of a motivic space (as opposed to a motivic spectrum):

 \begin{rmk}[Instability condition for $\R$-motivic cohomology] 
 \label{Instab}
 If $X$ is an $\R$-motivic space then $\mr{H}^{*,*}(X)$ admits a ring structure, and, 
 for any $u \in  \mr{H}^{n, i}(X)$,
 the $\R$-motivic squaring operations obey the rule 
\[ 
\Sq^{2i}(u) = \left\lbrace 
\begin{array}{ccc}
0 & \text{if $n< 2i$} \\
u^2 & \text{if $n = 2i$.}
\end{array}
\right.
\]
This is often  referred to as the \emph{instability condition}. 
\end{rmk}

 To understand the $\cA^\R$-module structure of $\mr{H}^{*, *}(\QR)$, we first make the following observation regarding $\mr{H}^{*, *}(\coneR{\hsf})$ (as $\coneR{\hsf}$ is a sub-complex of $\QR$) using an argument very similar to \cite[Lemma~7.4]{LowMilnorWitt}.

 \begin{prop}\label{A(0)exnts}
 There are two extensions of $\cAR(0)$ to an $\cAR$-module, and these $\cAR$-modules are realized as the cohomology of $\coneR{\hsf}$ and $\coneR{2}$.
\end{prop}

 \begin{figure}[h]
\vspace{1ex}

\includegraphics[width=0.4\textwidth]{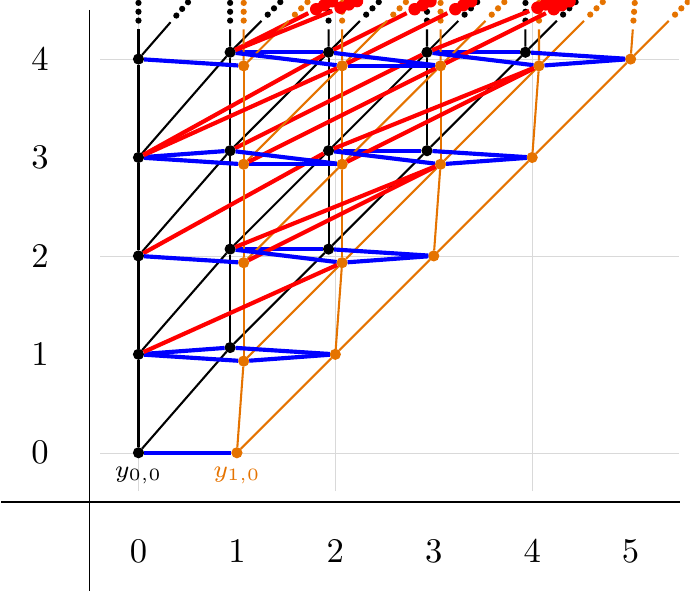} \qquad
\includegraphics[width=0.4\textwidth]{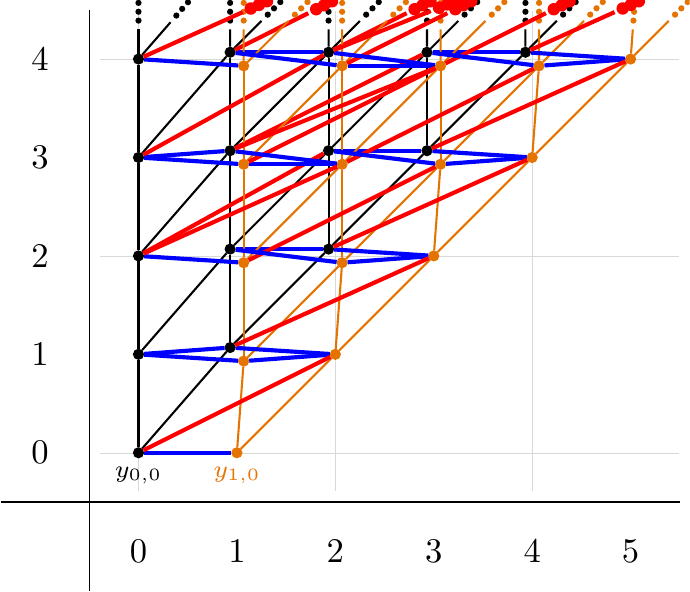} \\

$\rH^{*, *}_\R(\coneR{\hsf})$ \qquad \qquad \qquad \qquad \qquad \qquad $\rH^{*, *}_\R(\coneR{2})$
\caption{The $x$-axis represents the negative of topological dimension, $y$-axis represents the negative of motivic weight, vertical lines of length $(0,1)$ represent $\tau$-multiplication, diagonal lines of length $(1,1)$ represent $\rho$-multiplication, blue lines represent $\Sq^1$-action and red lines represent $\Sq^2$-action.}
\label{fig:Smod2andSmodh}

\end{figure}

\begin{pf}
For degree reasons, the only choice in extending $\cAR(0)$ to an $\cAR$-module is the action of $\Sq^2$ on the generator in bidegree $(0, 0)$. 
Writing $y_{0,0}$ for the generator in degree (0, 0) and $y_{1,0}$ for $\Sq^1(y_{0,0})$ in (cohomological) bidegree $(1,0)$. 
The two possible choices are 
\begin{itemize}
\item $\Sq^2(y_{0,0})=0$ and
\item $\Sq^2(y_{0,0}) = \rho\cdot y_{1,0}$.
\end{itemize}

We can realize the degree 2 map  as an unstable map $\mr{S}^{1,0} \rtarr \mr{S}^{1,0}$, and we will write $\coneR{2}_u$ for the cofiber.  We deduce information about the $\cA^\R$-module structure of $\mr{H}^{\ast, \ast}(\coneR{2})$ by   analyzing  the  cohomology ring  of $\mr{S}^{1,1} \sma \coneR{2}_u$ using the instability condition of \autoref{Instab}.
First, note that  in
\[ \mr{H}^{\ast, \ast}(\mr{S}^{1,1} ) \cong \mathbb{M}_2^\R  \cdot \iota_{1,1}\]
we have the relation $\iota_{1,1}^2 = \rho \cdot \iota_{1,1}$ \cite[Lemma~6.8]{V}. Also note that
\[ \mr{H}^{*,*}((\coneR{2}_u)_+) \cong \mathbb{M}_2^\R[x]/(x^3) \]
  where  $x$   is in cohomological degrees $(1,0)$. 
 Therefore, in 
\[ \mr{H}^{\ast, \ast}(\mr{S}^{1,1} \sma \coneR{2}_u ) = \mathbb{M}_2^\R  \cdot \iota_{1,1} \otimes_{\mathbb{M}_2^\R} \mathbb{M}_2^\R \{ x, x^2 \}   \]
the instability condition implies
\[ \Sq^2 ( \iota_{1,1} \otimes  x) = \iota_{1,1}^2 \otimes x^2 = \rho \cdot \iota_{1,1} \otimes x^2.\]
  Here the space-level cohomology class $x^2$ corresponds to the spectrum-level class $y_{1,0}$.
 Therefore, $\Sq^2(y_{0,0}) = \rho \cdot y_{1,0}$ in $\mr{H}^{*,*}(\coneR{2})$. This is also reflected in the fact that multiplication by $2$ is detected by $h_0 + \rho h_1$ in the $\R$-motivic Adams spectral sequence \cite[$\mathsection$8]{LowMilnorWitt}. 

On the other hand $\hsf$ is the `zeroth $\R$-motivic Hopf-map' detected by the element $h_0$ in the motivic Adams spectral sequence. It follows that $\Sq^2(y_{0,0}) = 0$. 
\end{pf}

 In order to express the $\cA^\R$-module structure on $\mr{H}^{\ast, \ast}(X)$ for a finite spectrum $X$, it is enough to specify the action of $\cA^\R$ on its left $\mathbb{M}_2^\R$-generators as the action of $\tau$ and $\rho$ multiples are determined by the Cartan formula. 

 \begin{eg}
  Let $ \{ y_{0,0}, y_{1,0}\} \subset \mr{H}^{*,*}( \coneR{\hsf})$ denote a left $\mathbb{M}_2^\R$-basis of $\mr{H}^{*,*}( \coneR{\hsf})$.
  The data that 
 \begin{itemize}
 \item $\Sq^1(y_{0,0}) = y_{1,0}$ 
 \item $\Sq^2(y_{0,0}) = 0$ 
 \end{itemize}
completely determines the $\cA^\R$-module structure of  $\mr{H}^{*,*}( \coneR{\hsf})$.
\end{eg}

\begin{prop} \label{lem:Qdescribe}$\mr{H}^{\ast, \ast}(\QR)$ is a free $\mathbb{M}_2^{\R}$-module  generated by $a, b$ and $c$ in cohomological bidegrees $(0,0), (1,0)$ and $(3,1)$, and the relations 
\begin{enumerate}
\item  $\Sq^1(a) = b$, 
\item $\Sq^2(b) = c$, 
\item $\Sq^4(a) = 0$.   
\end{enumerate}
completely determine the $\cA^{\R}$-module structure  of $\mr{H}^{\ast, \ast}(\QR)$. 
\end{prop}

 \begin{proof}
 $\mr{H}^{\ast, \ast}(\QR)$ is a free $\mathbb{M}_2^{\R}$-module because the attaching maps of  $\QR$ induce trivial maps in $\mr{H}^{*,*}(-)$. 
  The first two relations can be deduced from the obvious maps 
 \begin{enumerate}
 \item  $\coneR{\hsf} \to \QR$ 
 \item $ \QR \to \Sigma^{1,0}\,\coneR{\eta_{1,1}}$
 \end{enumerate}
 which are respectively surjective and injective in cohomology.
 
  Let $\hsf^u: \mr{S}^{3,2} \to  \mr{S}^{3,2} $ and $\eta_{1,1}^u: \mr{S}^{3,2} \to \mr{S}^{2,1}$ denote the unstable maps that stabilize to $\hsf$ and $\eta_{1,1}$, respectively. 
 The unstable $\R$-motivic space $\QR^u$ (which stabilizes to $\QR$)  can be constructed  using the fact that the composite of the unstable maps 
 \[
 \begin{tikzcd}
  \mr{S}^{4,3} \rar["\Sigma^{1,1}\eta_{1,1}^u"] & \mr{S}^{3,2} \rar[" \hsf^u"] & \mr{S}^{3,2}
  \end{tikzcd}
  \]

  is null. Thus $\mr{H}^{*, *}(\QR^u )$ consists of three generators $a_{u}$, $b_u$ and $c_u$ in bidegrees $(3,2)$, $(4,2)$ and $(6,3)$. It follows from the instability condition that $\Sq^4(a_u) = 0$. 
 \end{proof}
  \begin{proof}[Proof of \autoref{thm:A1R}] From \autoref{rem:A00} 
  and \autoref{lem:splitfixA}, we deduce that $\cA_1^\R$ is a type $(2,1)$ complex.  To  show that  the bi-graded $\R$-motivic cohomology of $\cA_1^\R$ is free as an $\cA^\R(1)$, we make use of \autoref{cor:freeR}. 
  
   Since  $\mr{H}^{*, *}( \cA_1^\R)$ is a summand of a free $\mathbb{M}_2^\R$-module, it is projective as an $\mathbb{M}_2^\R$-module. In fact, $\mr{H}^{*, *}( \cA_1^\R)$ is free, as projective modules over (graded) local rings are free. Also note that the 
   elements
    \[ \overline{\mr{P}}^0_1, \overline{\mr{P}}^1_1, \overline{\mr{P}}^0_2 \in  \cA^{\R}(1)/ (\rho, \tau) \cong \Lambda( \overline{\mr{P}}^0_1, \overline{\mr{P}}^1_1, \overline{\mr{P}}^0_2) \]
    are primitive. Hence we have a Kunneth isomorphism in the respective Margolis homologies, in particular we have, 
     \[ \mathcal{M}( \mr{H}^{*, *}(\cA^\R_1)/(\rho, \tau), \overline{\mr{P}}^s_t) = \overline{e}( \mathcal{M}( \mr{H}^{*, *}(\QR)/(\rho, \tau), \overline{\mr{P}}^s_t )^{\otimes 3} ) \]
for $(s,t) \in \{ (0,1), (1,1), (0,2) \}$. Since $ \dim_{\F} \mathcal{M}(\mr{H}^{*, *}( \QR)/(\rho, \tau), \overline{\mr{P}}^s_t ) = 1$, by  \linebreak \autoref{lem:vanish} 
\[ \mathcal{M}( \cA^\R_1/(\rho, \tau), \overline{\mr{P}}^s_t) =0 \] 
for $(s,t) \in \{ (0,1), (1,1), (0,2) \}$. Thus, by \autoref{cor:freeR} we conclude that $\mr{H}^{*, *}( \cA_1^\R)$ is a free $\cA^\R(1)$-module. A direct computation shows that 
\[ \dim_{\F} \mr{H}^{*, *}(\cA^\R_1)/(\rho, \tau) = 8,\] hence $\mr{H}^{*, *}( \cA_1^\R)$ is $\cA^\R(1)$-free
of rank one.
  \end{proof}
  
  \subsection{The $\cAR$-module structure}\label{cARmod-subsec}
 Using  the description \eqref{eqn:identifycohomology} and Cartan formula we make a complete calculation of the  $\cA^\R$-module structure of $\mr{H}^{*,*}(\cA_1^\R)$. Let $a, b, c \in \mr{H}^{*,*}(\QR)$ as in \autoref{lem:Qdescribe}. In \autoref{fig:motivicA1} we provide a pictorial representation with the names of the generators that are in the image of the idempotent $\overline{e}$. 
 For convenience we relabel 
 the generators in \autoref{fig:motivicA1}, where the indexing on a new label records the cohomological bidegrees of the corresponding generator. The following result is straightforward, and we leave it to the reader to verify. 
 
 \begin{lem} \label{lem:AA_1} In $\mr{H}^{\ast, \ast}(\cA_1^\R)$, the underlying $\cA^\R(1)$-module structure, along with the  relations 
 \begin{enumerate}
 \item $\Sq^4 (v_{0,0}) = \tau \cdot w_{4,1}$, 
 \item $\Sq^4 (v_{1,0}) = w_{5,2}$, 
 \item $ \Sq^4 (v_{2,1}) = 0$,
 \item $ \Sq^4( v_{3,1}) = 0 =  \Sq^4( w_{3,1}) $,
 \item $ \Sq^8 (v_{ 0, 0}) = 0$,
  \end{enumerate}
 completely determine the $\cA^\R$-module structure.
 \end{lem}
 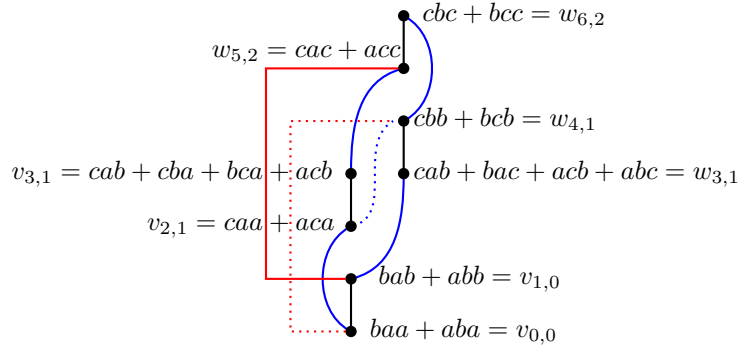
\begin{figure}[h]
 \[
\begin{tikzpicture}\begin{scope}[thick, every node/.style={sloped,allow upside down}, scale=0.7]
\draw (0,0)  node[inner sep=0] (v00) {} -- (0,1) node[inner sep=0] (v01) {};
\draw (0,2)  node[inner sep=0] (v11) {} -- (0,3) node[inner sep=0] (v12) {};
\draw (1,3)  node[inner sep=0] (v22) {} -- (1,4) node[inner sep=0] (v23) {};
\draw (1,5)  node[inner sep=0] (v33) {} -- (1,6) node[inner sep=0] (v34) {};
 \draw [color=blue] (v00) to [out=150,in=-150] (v11);
 \draw [color=blue] (v01) to [out=15,in=-90] (v22);
 \draw [color=blue] (v12) to [out=90,in=-165] (v33);
 \draw [color=blue] (v23) to [out=30,in=-30] (v34);
 \draw [dotted][color=blue] (v11) to [out=15,in=-150] (.8, 4);
\filldraw (v00) circle (2.5pt);
\filldraw (v01) circle (2.5pt);
\filldraw (v11) circle (2.5pt);
\filldraw (v12) circle (2.5pt);
\filldraw (v22) circle (2.5pt);
\filldraw (v23) circle (2.5pt);
\filldraw (v33) circle (2.5pt);
\filldraw (v34) circle (2.5pt);
\draw (0,0) node[right]{$\ baa+aba = v_{0,0}$} (0,1) node[right]{$\ \ bab+abb=v_{1,0}$} (0,2) node[left,xshift={1pt}]{$v_{2,1} = caa+aca \ $ } (0,3) node[left]{$v_{3,1} = cab+cba+bca+acb$ \ };
\draw (1,6) node[right]{$\ cbc+bcc=w_{6,2}$} (1.3,5.3) node[left]{$w_{5,2} = cac+acc \ $} (1,4) node[right]{$cbb+bcb = w_{4,1} $} (1,3) node[right]{$cab+bac+acb+abc = w_{3,1}$};
\draw[color = red] (-.1,1) -- (-1.615, 1 )-- (-1.615, 5) -- (.9, 5);
\draw[dotted][color = red] (-.1,0) -- (-1.15, 0 )-- (-1.15, 4) -- (.8, 4);
\end{scope}\end{tikzpicture}
\]
\caption{ We depict the $\cA^\R$-module structure of $\mr{H}^{*,*}(\cA_1)$. The black, blue, and red lines represent the action of  motivic $\Sq^1$, $\Sq^2$, and $\Sq^4$, respectively. Black dots represent $\mathbb{M}_2^\R$-generators,  and a dotted line represents that the action hits the $\tau$-multiple of the given $\mathbb{M}_2^\R$-generator. }
\label{fig:motivicA1}
 \end{figure}

\begin{rmk} \label{rem:A128} In upcoming work, we show that $\cA^\R(1)$ admits $128$ different $\cA^\R$-module structures.  Whether all of the $128$ $\cA^\R$-module structures can be realized by  $\R$-motivic spectra, or not, is currently under investigation.
\end{rmk}
\section{An $\R$-motivic $v_1$\slfmp} \label{selfmap}

With the construction of $\cA_1^\R$, we hope that  any one of $\Y_{(i,j)}^\R$ fits into an exact triangle  
\begin{equation} \label{eqn:exact}
\begin{tikzcd}
\Sigma^{2,1} \mathcal{Y}^\R_{(i,j)} \rar["v"] & \mathcal{Y}^\R_{(i,j)} \rar[] & \cA_1^\R \rar & \Sigma^{3,1} \mathcal{Y}^\R_{(i,j)} \rar["\Sigma v"] & \dots 
\end{tikzcd}
\end{equation}
in $\ho(\Sp_{2, \fin}^{\R} )$. The motivic weights prohibit $\cAR_1$ from being the cofiber of a self-map on $\Ytriv$ or $\Ytwo$. We will also see  that the spectrum $\Yone^\R$ cannot be a part of  \eqref{eqn:exact} because of its $\cA^\R$-module structure (see \autoref{lem:YSq4}).
 If $\Y_{(i,j)} = \Ythree^\R$ in \eqref{eqn:exact}, then the map $v$ will necessarily be  a $v_{1, \mr{nil}}$\slfmp\ because $\Ythree^\R$  is  of type $(1,1)$ and  $\cA_1^\R$ is of type $(2,1)$.  The main purpose of this section is to prove \autoref{main1} and \autoref{thm:cofibv1} by showing  that  $\Ythree^\R$ does fit into an exact triangle very similar to  \eqref{eqn:exact}
 \[ 
 \begin{tikzcd}
\Sigma^{2,1} \mathcal{Y}^\R_{(i,j)} \rar["v"] & \mathcal{Y}^\R_{(i,j)} \rar[] & \coneR{v} \rar & \Sigma^{3,1} \mathcal{Y}^\R_{(i,j)} \rar["\Sigma v"] & \dots 
\end{tikzcd}
 \]
 where $\coneR{v}$  is of type $(2,1)$  and  $\mr{H}^{*,*}(\coneR{v}) \cong \mr{H}^{*,*}(\cA_1^\R)$ as $\cA^\R$-modules. 
 
 \begin{rmk}
  The fact that $\mr{H}^{*,*}(\coneR{v})$ is isomorphic to $\mr{H}^{*,*}(\cA_1^\R)$ as  $\cA^\R$-modules does not 
 imply that $\coneR{v}$ and $\cA_1^\R$ are equivalent as $\R$-motivic spectra.
  There are a plethora of examples of Steenrod modules that are realized by spectra of different homotopy types.  
  \end{rmk} 
   
 We begin  by discussing the $\cA^\R$-module structures of $\mr{H}^{*,*}( \Ythree^\R)$. Using Adem relations, one can show that the element 
\[ Q_1 := \Sq^1 \Sq^2 + \Sq^2 \Sq^1 \in \cA^\R(1) \] 
squares to zero. Let $\Lambda(Q_1)$ denote the exterior subalgebra $\mathbb{M}_2^\R[Q_1]/(Q_1^2)$ of $\cA^\R(1)$. Let $\mathcal{B}^\R(1)$ denote the $\cA^\R(1)$-module 
\[ 
\mathcal{B}^\R(1) := \cA^\R(1) \otimes_{\Lambda(Q_1)} \mathbb{M}_2^\R. 
\]
Both $\Yone^\R$ and $\Ythree^\R$ are realizations of $\mathcal{B}^\R(1)$. In other words: 

\begin{prop} There is an isomorphism of $\cA^\R(1)$-modules 
\[  \mr{H}^{*,*}(\cY_{(i, j)}^\R) \cong \mathcal{B}^\R(1) \]
for $(i,j) \in \{ (2,1), (\hsf,1)\}$. 
\end{prop}

\begin{proof} 
Note that $\mr{H}^{*,*}(\mathcal{Y}_{(i,j )}^\R)$ is  cyclic as an $\cA^\R(1)$-module for $(i,j) \in \{ (2,1), (\hsf,1)\}$. Thus we have an $\cA^\R(1)$-module map 
\begin{equation} \label{fi}
 f_i: \cA^\R(1) \to \mr{H}^{*,*}(\mathcal{Y}_{(i, j)}^\R).
 \end{equation}
The result follows from the fact that $Q_1$ acts trivially on $\mr{H}^{*,*}(\Y_{(i, j)}^\R)$ and a dimension counting argument. 
\end{proof}
\begin{rmk} \label{rem:mapfi} Let  $\{ y_{0,0}, y_{1,0} \}$  be the $\mathbb{M}_2^\R$-basis of  $\mr{H}^{*,*}( \coneR{\hsf})$ or $\mr{H}^{*,*}( \coneR{2})$, so that $\Sq^1(y_{0,0}) = y_{1,0}$, and let $\{ x_{0,0}, x_{2,1} \}$ a basis of $\coneR{\eta_{1,1}}$, so that $\Sq^2(x_{0,0}) = x_{2,1}$.
If we consider the $\mathbb{M}_2^\R$-basis $\{ v_{0,0}, v_{1,0}, v_{2,1}, v_{3,1}, w_{3,1}, w_{3,2}, w_{4,2}, w_{5,3}, w_{6,3} \}$ of $\cA^\R(1)$ from \autoref{cARmod-subsec}, then the maps $f_i$ of \eqref{fi} are given as in \autoref{tbl:fimaps}.
\begin{table}[ht]
\captionof{table}{The maps $f_2$ and $f_\hsf$}
\label{tbl:fimaps}
\begin{center}
\begin{tabular}{LLL} 
\hline
x & f_2(x) & f_\hsf(x)   \\ \hline  
v_{0,0} & y_{0,0}x_{0,0} & y_{0,0}x_{0,0} \\
v_{1,0}  & y_{1,0} x_{0,0} & y_{1,0} x_{0,0} \\
v_{2,1} &  y_{0,0} x_{2,0} + \rho \cdot y_{1,0} x_{0,0} &  y_{0,0} x_{2,0} \\
v_{3,1} & y_{1,0} x_{2,0} & y_{1,0} x_{2,0} \\
w_{3,1} & y_{1,0} x_{2,0} & y_{1,0} x_{2,0} \\
w_{4,2} & 0 & 0 \\
w_{5,3} & 0 & 0 \\
w_{6,3} & 0 & 0 \\
\hline
\end{tabular}
\end{center}
\end{table}

\end{rmk}

\begin{lem} \label{lem:YSq4}The $\cA^\R$-module structures on $\mr{H}^{*,*}(\Yone^\R)$ and $\mr{H}^{*,*}(\Ythree^\R)$ are 
given as in \autoref{fig:YoneYthree}.
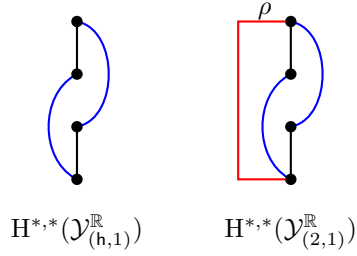
\begin{figure}[h]
\begin{tikzpicture}\begin{scope}[thick, every node/.style={sloped,allow upside down}, scale=0.7]
\draw (0,0)  node[inner sep=0] (v00) {} -- (0,1) node[inner sep=0] (v01) {};
\draw (0,2)  node[inner sep=0] (v11) {} -- (0,3) node[inner sep=0] (v12) {};
 \draw [color=blue] (v00) to [out=150,in=-150] (v11);
 \draw [color=blue] (v01) to [out=15,in=-15] (v12);
\filldraw (v00) circle (2.5pt);
\filldraw (v01) circle (2.5pt);
\filldraw (v11) circle (2.5pt);
\filldraw (v12) circle (2.5pt);
\node at (0,-1) {$\mr{H}^{*,*}(\Ythree^\R)$};
\end{scope}\end{tikzpicture}
\qquad
\begin{tikzpicture}\begin{scope}[thick, every node/.style={sloped,allow upside down}, scale=0.7]
\draw (0,0)  node[inner sep=0] (v00) {} -- (0,1) node[inner sep=0] (v01) {};
\draw (0,2)  node[inner sep=0] (v11) {} -- (0,3) node[inner sep=0] (v12) {};
 \draw [color=blue] (v00) to [out=150,in=-150] (v11);
 \draw [color=blue] (v01) to [out=15,in=-15] (v12);
\filldraw (v00) circle (2.5pt);
\filldraw (v01) circle (2.5pt);
\filldraw (v11) circle (2.5pt);
\filldraw (v12) circle (2.5pt);
\draw[][color = red] (-.1,0) -- (-1, 0 )-- (-1, 3) -- (-.1, 3);
\node at (-0.5,3.2) {$\rho$};
\node at (0,-1) {$\mr{H}^{*,*}(\Yone^\R)$};
\end{scope}\end{tikzpicture}
\caption{ Black, blue, and red lines represent the action of   $\Sq^1$, $\Sq^2$, and $\Sq^4$, respectively. Black dots represent $\mathbb{M}_2^\R$-generators,  and in the case of $\Yone^\R$, $\Sq^4$ on the bottom cell is $\rho$ times the top cell. }
\label{fig:YoneYthree}
 \end{figure}
\end{lem}

\begin{proof} The result is an easy consequence of a calculation using the Cartan formula
\[ \Sq^4 (xy) = \Sq^4(x)y +\tau \Sq^3(x) \Sq^1(y) + \Sq^2(x)  \Sq^2(y) + \tau\Sq^1 (x) \Sq^3(y) + x \Sq^4(y) \]
and the fact that $\Sq^2( y_{0,0}) = \rho y_{1,0}$ in $\mr{H}^{*,*}(\coneR{2})$, whereas  $\Sq^2( y_{0,0}) = 0$ in $\mr{H}^{*,*}(\coneR{\hsf})$ (see \autoref{A(0)exnts}).
\end{proof}

\begin{rmk} Comparing \autoref{lem:YSq4} and \autoref{lem:AA_1},  we see that the $\cA^\R(1)$-module map $f_2$, as in \autoref{rem:mapfi}, cannot be extended to a map of  $\cA^\R$-modules. 
\end{rmk}

\begin{cor} There is an exact sequence of $\cA^\R$-modules 
\begin{equation} \label{vexact}
\begin{tikzcd}
0 \rar & \mr{H}^{*,*}(\Sigma^{3,1} \Ythree^\R) \rar["\pi^*"] & \mr{H}^{*,*}(\cA_1^\R) \rar["\iota^*"] & \mr{H}^{*,*}(\Ythree^\R)   \rar & 0 .
\end{tikzcd}
\end{equation}
\end{cor}

\begin{proof} From the description of the map $f_\hsf$ in \autoref{rem:mapfi}, along with \autoref{lem:AA_1} and  \autoref{lem:YSq4}, it is easy to check that $f_\hsf$ extends to an $\cA^\R$-module map and that 
\[ \ker f_\hsf \cong \mr{H}^{*,*}(\Ythree^\R)\]
as $\cA^\R$-modules. 
\end{proof}

The exact sequence \eqref{vexact} corresponds to a nonzero element in the $\mr{E}_2$-page of the $\R$-motivic Adams spectral sequence  (also see \autoref{rem:algdual} and \autoref{dualASS} )
\begin{equation} \label{v:ASS}
 \overline{v} \in \Ext_{\cA^\R}^{2,1,1}(  \mr{H}^{*,*}(\Ythree^\R \sma D \Ythree^\R), \mathbb{M}_2^\R   ) \Rightarrow [ \Ythree^\R, \Ythree^\R]_{2, 1},
 \end{equation}
where $D \Ythree^\R := \mr{F}(\Ythree^\R, \bS_\R )$ is the Spanier-Whitehead dual of $\Ythree^\R$. If 

\begin{notn} 
Note that we follow \cites{LowMilnorWitt,BI} in grading $\Ext_{\cAR}$ as $\Ext_{\cAR}^{s,f,w}$, where $s$ is the stem, $f$ is the Adams filtration, and $w$ is the weight. We will also follow \cite{C2MW0} in referring to the difference $s-w$ as the {\it coweight}. 
\end{notn}

\begin{rmk} \label{rem:algdual}
Since $\mr{H}^{*,*}( \Ythree^\R)$ is $\mathbb{M}_2^\R$-free, an appropriate universal-coefficient spectral sequence collapses and we get 
$\mr{H}^{*,*}( D\Ythree^\R)  \cong \hom_{\mathbb{M}_2^\R}( \mr{H}^{*,*}( \Ythree^\R), \mathbb{M}_2^\R). $
Further, the Kunneth isomorphism of \eqref{eqn:Kunnethiso} gives us 
\[ \mr{H}^{*,*}(\Ythree^\R \sma D \Ythree^\R) \cong \mr{H}^{*,*}(\Ythree^\R) \otimes_{\mathbb{M}_2^\R} \mr{H}^{*,*}(D\Ythree^\R), \]
and therefore, 
\[ \Ext_{\cA^\R}^{*,*,*}( \mathbb{M}_2^\R , \mr{H}^{*,*}(\Ythree^\R \sma D \Ythree^\R)  ) \cong \Ext_{\cA^\R}^{*,*,*}(\mr{H}^{*,*}(\Ythree^\R) , \mr{H}^{*,*}(\Ythree^\R)  ) .\]
\end{rmk}

 \autoref{main1} follows immediately if we show that the element $\overline{v}$ is a nonzero permanent cycle.  The following lemma implies  that a $d_r$-differential  (for $r \geq 2$) supported by  $\overline{v}$  has no potential nonzero target. 
 
\begin{prop} \label{prop:notarget} For $f \geq 3$,  $ \Ext_{\cA^\R}^{1, f, 1}(\mr{H}^{*,*}(\Ythree^\R) , \mr{H}^{*,*}(\Ythree^\R)  )  = 0 $.
\end{prop}

In order to calculate $\Ext_{\cA^\R}^{*,*,*}(\mr{H}^{*,*}(\Ythree^\R) , \mr{H}^{*,*}(\Ythree^\R)  )$, we filter the spectrum $\Ythree^\R$ via the evident maps
\[  
\begin{tikzcd}
Y_0 \dar[equals] \rar &  Y_1 \dar[equals]  \rar  &  Y_2 \dar[equals]  \rar & Y_3. \dar[equals]    \\
\bS_\R &  \coneR{\hsf} & \coneR{\hsf} \cup_{\bS_\R} \coneR{\eta_{1,1}} & \Ythree^\R
\end{tikzcd}
\]
Note that $\mr{H}^{*,*}(Y_j)$ are free $\mathbb{M}_2^\R$-modules. The above filtration results in cofiber  sequences 
\[ 
\begin{tikzcd}[row sep=small]
Y_0 \rar & Y_1 \rar & \Sigma^{1,0} \bS_\R, &\\
Y_1 \rar & Y_2 \rar & \Sigma^{2,1} \bS_\R, & \text{and} \\
Y_2 \rar & Y_3  \rar & \Sigma^{3,1} \bS_\R,  &
\end{tikzcd}
\] 
which induce short exact sequences of $\cA^\R$-modules as the connecting map 
\[ \coneR{Y_j \to Y_{j+1}} \longrightarrow \Sigma Y_j \]
induces the zero map in $\mr{H}^{*,*}(-)$. Thus, applying the functor $\Ext_{\cA^\R}^{*,*,*}(  \mr{H}^{*,*}(\Ythree), -)$ to these short-exact sequences, we get long exact sequences, which can be spliced together to obtain an Atiyah-Hirzebruch like  spectral sequence 
\[ 
\begin{tikzcd}
\mr{E}_1^{*,*,*,*} = \Ext_{\cA^\R}^{*,*,*}(\mr{H}^{*,*}(\Ythree),  \mathbb{M}_2^\R) \{ g_{0,0}, g_{1,0}, g_{2, 1 }, g_{3, 1} \} \dar[Rightarrow]   \\
\Ext_{\cA^\R}^{*,*,*}(\mr{H}^{*,*}(\Ythree^\R) , \mr{H}^{*,*}(\Ythree^\R)  ).
\end{tikzcd}
\]
An element  $x \cdot g_{i,j}$ in the $\mr{E}_2$-page contributes to the degree $ |x| -(i,0,j)$ of the abutment. Thus, \autoref{prop:notarget} is a straightforward consequence of the following \autoref{prop:notarget2}. 

\begin{rmk} \label{dualASS} Because, $\mr{H}^{*,*}(\Ythree^\R)$ is $\mathbb{M}_2^\R$-free and finite, we have 
\[ \mr{H}_{*,*}( \Ythree^\R ) \cong \hom_{\mathbb{M}_2^\R}( \mr{H}^{*,*}(\Ythree), \mathbb{M}_2^\R ), \]
and therefore,
$\Ext_{\cA^\R}^{s,f, w}( \mr{H}^{*,*}(\Ythree^\R), \mathbb{M}_2^\R) \cong  \Ext_{\cA^\R_*}^{s,f,w}(\mathbb{M}_2^\R, \mr{H}_{*,*}(\Ythree^\R)).$
\end{rmk}

\begin{prop} \label{prop:notarget2} For $f \geq 3$ and $(i,j) \in \{ (0,0), (1,0), (2,1), (3,1) \}$, we have that
 \[  \Ext_{\cA^\R_*}^{ 1 +i, f, 1 +j}( \mathbb{M}_2^\R,\mr{H}_{*,*}(\Ythree^\R)  )  = 0 .\]
\end{prop}

\begin{proof}
Our desired vanishing concerns only the groups $\Ext_{\cAR_*}(\bMR, \mr{H}_{*,*}(\Ythree^\R)  )$
in coweights 0, 1 and 2. 
 These groups  can be easily  calculated  starting from the computations of $\Ext_{\cAR_*}^{*,*,*}(\bMR, \bMR)$ in \cite{LowMilnorWitt} and \cite{BI} and using the short exact sequences in $\Ext_{\cAR_*}$ arising from the cofiber sequences
\[
\Sigma^{1,1} S_\R \xrtarr{\eta_{1,1}} S_\R \rtarr \coneR{\eta_{1,1}} \quad \text{and} \] \[ \coneR{\eta_{1,1}} \xrtarr{\hsf} \coneR{\eta_{1,1}} \rtarr \coneR{\hsf} \sma \coneR{\eta_{1,1}} = \Ythree^{ \R}.
\]
We display $\Ext_{\cAR_*}(\bMR, \mr{H}_{*,*}(\coneR{\eta_{1,1}}))$ in coweights 0, 1 and 2 in the charts below. 
\vspace{5pt} \\
\centerline{
\includegraphics[width=0.4\textwidth]{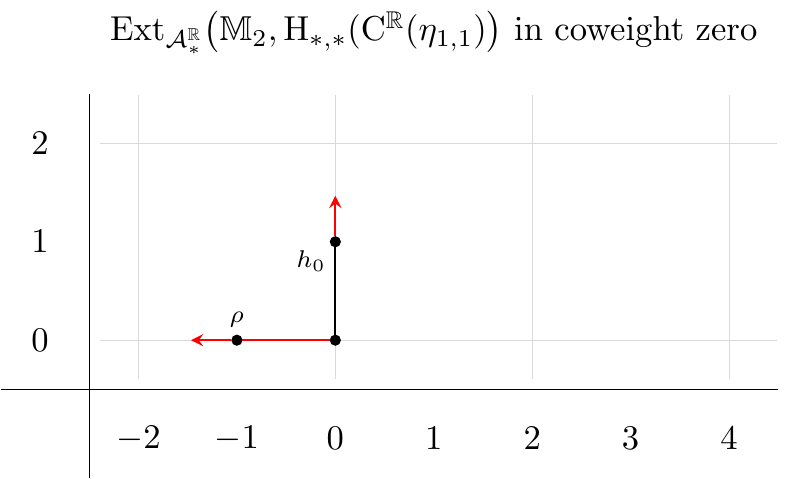} \qquad 
\includegraphics[width=0.4\textwidth]{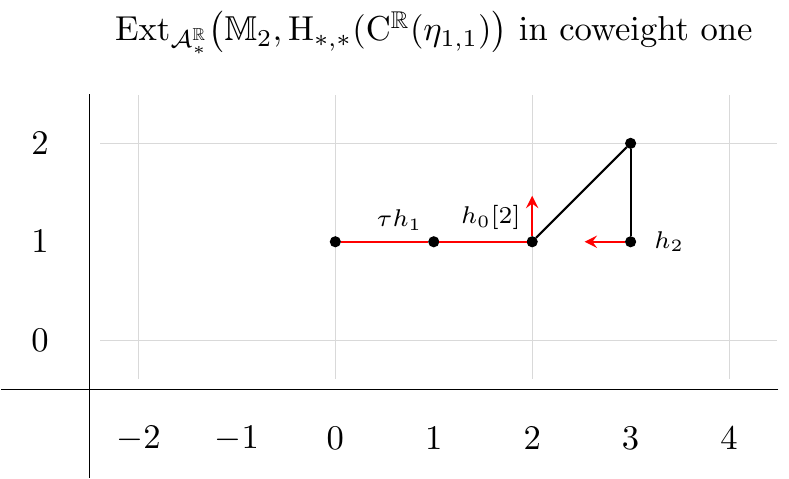} 
}
 \centerline{
\includegraphics[width=0.5\textwidth]{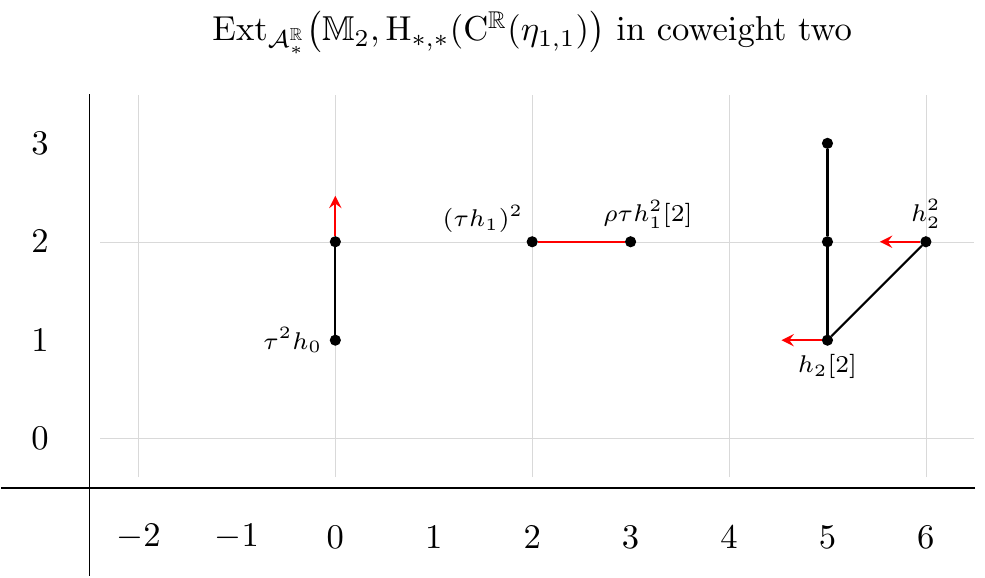} 
} \vspace{-1pt} \\
We find that $\Ext_{\cAR_*}(\bMR, \mr{H}_{*,*}(\Ythree^\R)  )$ is, in coweights zero,  one, and two, also given by the charts below. \vspace{5pt} \\
\centerline{
\includegraphics[width=0.4\textwidth]{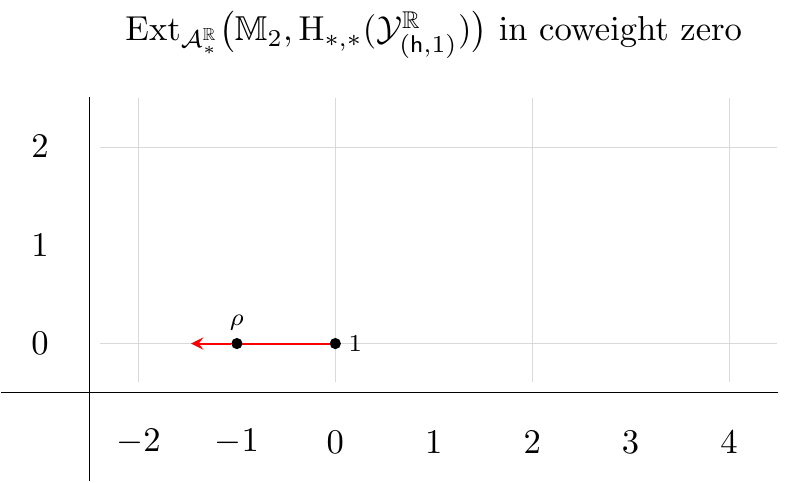} \qquad 
\includegraphics[width=0.4\textwidth]{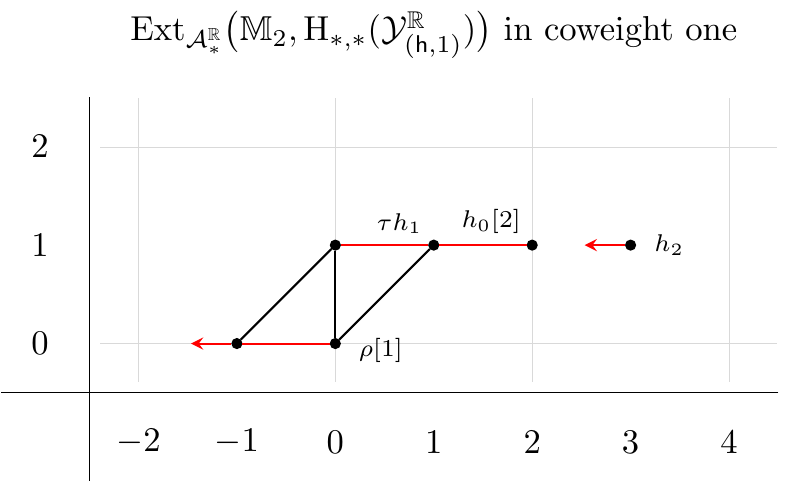} 
} \centerline{
\includegraphics[width=0.5\textwidth]{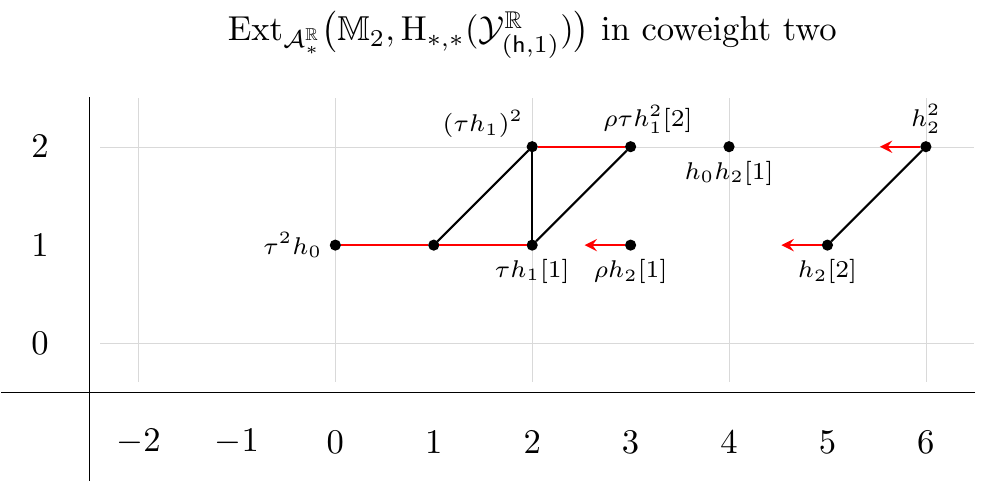} 
}
The result follows from the above charts. 
\end{proof}
\begin{rmk} One can also resolve \autoref{prop:notarget2}  directly using the $\rho$-Bockstein spectral sequence 
\begin{equation} \label{rhoB}
\begin{tikzcd}
\mr{E}_1:=   \Ext_{\cA^\C_\ast}(\F[\tau], \mr{H}_{*,*}(\Ythree^\C)) \otimes \F[\rho]  \dar[Rightarrow] \\ 
\Ext_{\cA^\R_\ast}(\mathbb{M}_2^\R, \mr{H}_{*,*}(\Ythree^\R)) 
\end{tikzcd}
\end{equation}
and identifying a vanishing region for $\Ext_{\cA^\C_\ast}^{s,f,w}(\F[\tau], \mr{H}_{*,*}(\Ythree^\C))$. Even a rough estimate of the vanishing region using the $\mr{E}_1$-page of the  $\C$-motivic May spectral sequence 
leads to \autoref{prop:notarget2}. Such an approach would avoid explicit calculations of $\Ext_{\cAR}$ as in \cite{LowMilnorWitt} and \cite{BI}.
\end{rmk}

\begin{pf}[Proof of \autoref{main1}]Since  $ \text{\autoref{prop:notarget2}} \implies \text{\autoref{prop:notarget}}$, every map 
\[
\begin{tikzcd}
 v: \Sigma^{2,1}\Ythree^\R \rar & \Ythree^\R
 \end{tikzcd}
 \]
detected by $\overline{v}$ of \eqref{v:ASS} is a nonzero permanent cycle. In order to finish the proof of \autoref{main1} we must show that $v$ is necessarily $v_{(1,\nil)}$\slfmp\ of periodicity $1$. It is easy to see that the underlying map 
\[ 
\begin{tikzcd}
\Phi^e(\Betti(v)): \Sigma^2 \cY  \rar & \cY
\end{tikzcd}
\]
is a $v_1$\slfmp\ of periodicity $1$ as  \[ \mr{C}(\Phi^e(\Betti(v))) \simeq \Phi^e(\Betti (\coneR{v})) \simeq \cA_1[10] \]  is of type $1$ (see \autoref{rem:A00}). On the other hand, 
\[ 
\begin{tikzcd}
\Phi^{\Ctwo}(\Betti(v)):  \Sigma^2(\Sigma \M_2(1) \vee  \M_2(1)) \rar & \Sigma \M_2(1) \vee \M_2(1)
\end{tikzcd}
\]
is necessarily a nilpotent map because of \cite[Theorem~3(ii)]{Thick} and the fact that a  $v_1$\slfmp\ of $\M_2(1)$ has periodicity at least $4$ (see \cite{DM} for details) which lives in $[\M_2(1), \M_2(1)]_{8k}$ for $k\geq 1$.   
\end{pf}
\begin{pf}[Proof of \autoref{thm:cofibv1}] Since $v$ is a $v_{(1,\nil)}$\slfmp\ and $\Ythree^\R$ is of type $(1,1)$, it follows that $\coneR{v}$ is of type $(2,1)$. Moreover,
\[ \mr{H}^{*,*}(\coneR{v}) \cong \mr{H}^{*,*}(\cA_1^\R)\]
as $v$ is detected by $\overline{v}$  of \eqref{v:ASS} in the $\mr{E}_2$-page of the Adams spectral sequence. Thus, $\mr{H}^{*,*}(\coneR{v})$ is a free $\cA^\R(1)$-module on single generator. 
\end{pf}
\begin{rmk} \label{rem:Y1v1} It is likely that  realizing a different $\cA^\R$-module structure on $\cA^\R(1)$  as a spectrum (see also \autoref{rem:A128}) may lead to a $1$-periodic $v_1$\slfmp\ on $\Yone^\R$ as well as  on $\Yone^{\Ctwo}$.  We explore such possibilities in  upcoming work. 
\end{rmk}

\section{ Nonexistence of $v_{1,0}$\slfmp\ on $\coneR{\hsf}$ and $\Ytwo^\R$ }
\label{nonexistence}

Let $X$ be a finite $\R$-motivic spectrum and let $f \colon \Sigma^{i,j} X \to X$ be a map such that 
 \[ 
 \begin{tikzcd}
 \Phi^{\Ctwo}(\Betti(f)) \colon \Sigma^{i-j} \Phi^{\Ctwo}(\Betti(X)) \rar & \Phi^{\Ctwo}(\Betti(X)) 
 \end{tikzcd} \] 
  is a $v_0$\slfmp. Then  it must be the case that $i =j$, as $v_0$\slfmp{s} preserve dimension. Note that both $\coneR{\hsf}$ and $\Ytwo^\R$ are of type $(1,0)$.

   \begin{prop} The $v_1$\slfmp s of $\M_2(1)$ are not in the image of the underlying homomorphism
\[ \Phi^e \circ \Betti \colon [ \Sigma^{8k,8k} \coneR{\hsf}, \coneR{\hsf}]^{\R} \rtarr [\Sigma^{8k} \M_2(1), \M_2(1)]. \]
\end{prop}

\begin{proof} The minimal periodicity of a $v_1$\slfmp\ of $\M_2(1)$ is $4$. Let $v:\Sigma^{8k} \M_2(1) \to \M_2(1)$ be a $4k$-periodic $v_1$\slfmp.  It is well-known that the composite 
\begin{equation} \label{nonzero}
\begin{tikzcd}
\Sigma^{8k} \bS\rar[hook] &\Sigma^{8k} \M_2(1) \rar["v"] &   \M_2(1) \rar &  \Sigma^1 \bS
\end{tikzcd}
\end{equation}
is not null (and equals $P^{k-1}(8\sigma)$ where $P$ is a periodic operator given by the Toda bracket $\langle \sigma, 16, - \rangle$).

Suppose there exists $f: \Sigma^{8k,8k}\coneR{\hsf} \to \coneR{\hsf}$ such that $\Phi^{e} \circ \Betti (f) = v$. Then \eqref{nonzero} implies that the composition 
\begin{equation} \label{nonzero2}
\begin{tikzcd}
\Sigma^{8k,8k} \bS_\R \rar[hook] &\Sigma^{8k, 8k} \coneR{\hsf} \rar["v"] &   \coneR{\hsf} \rar &  \Sigma^{1,0} \bS
\end{tikzcd}
\end{equation}
is nonzero as the functor $\Phi^e \circ \Betti$ is additive. The composite  of the maps in \eqref{nonzero2} is a nonzero element of $\pi_{*,*}(\bS_\R)$ in negative  coweight.  This contradicts  the fact that $\pi_{*,*}(\bS_\R)$ is trivial in negative coweights \cite{LowMilnorWitt}.
\end{proof}

   \begin{prop} The $v_1$\slfmp s of $\Y$ are not in the image of the underlying homomorphism
\[ \Phi^e \circ \Betti \colon [ \Sigma^{2k,2k} \Ytwo^\R, \Ytwo^\R]^{\R} \rtarr [\Sigma^{8k} \Y,  \Y]. \]
\end{prop}

\begin{proof} Let $v: \Sigma^{2k} \Y \to \Y$ denote a $v_1$\slfmp\ of periodicity $k$.  Notice that the composite 
 \begin{equation} \label{nonzero3}
\begin{tikzcd}
\mr{S}^{2k} \rar[hook] &\Sigma^{2k} \Y \rar["v"] &   \Y  \rar &   \Y_{\geq 1}
\end{tikzcd}
\end{equation}
where $\Y_{\geq 1}$ is the first coskeleton, must be nonzero. If not, then $v$ factors through the bottom cell resulting in a map $\mr{S}^{2k} \to \Sigma^{2k}\cY \to \bS$ which induces an isomorphism in $\mr{K}(1)$-homology, contradicting the fact that $\bS$ is of type $0$. 

If $f: \Sigma^{2k,2k} \Ytwo^\R \to   \Ytwo^\R$  were a map such that $\Phi^e \circ \Betti (f) = v$, then \eqref{nonzero3} would force one  among the hypothetical composites $(A)$, $(B)$ or $(C)$ in the diagram 
\[ 
\begin{tikzcd}
\Sigma^{2k,2k} \bS_\R \rar[hook]& \Sigma^{2k,2k}  \Ytwo^\R \rar \arrow[dr, dashed] \arrow[ddr, dashed] &  \Ytwo^\R \rar[dashed, "p_3"] &  \Sigma^{3,0} \bS_\R & (A) \\ 
&&\mr{Fib}(p_3) \rar[dashed, "p_2"] \uar[hook] & \Sigma^{2,0} \bS_\R & (B) \\
&&\mr{Fib}(p_2) \rar[dashed, "p_1"] \uar[hook] & \Sigma^{1,0} \bS_\R & (C) \\
\end{tikzcd}
\]
to exist as a nonzero map, thereby contradicting the fact that $\pi_{*,*}(\bS_\R)$ is trivial in negative coweights.
\end{proof}  

\begin{rmk} \label{rem:Ctwov10} The above results do not preclude the existence of a $v_{1,0}$\slfmp\ on $\coneC{\hsf}$ and $\Ytwo^\Ctwo$. Forthcoming work \cite{C2Stems} of the second author and Isaksen shows that $8\sigma$ is in the image of $\Phi^e\colon \pi_{7,8}(\bS_{\Ctwo}) \rtarr \pi_7(\bS)$ and suggests that $\coneC{\hsf}$ supports a $v_{1,0}$\slfmp. 

\end{rmk}

\bibliographystyle{amsalpha}

\end{document}